\documentclass[11pt]{amsart}

\usepackage[margin=1in]{geometry}
\usepackage{hyperref}

\usepackage{amsmath, amscd, amsthm, amssymb}

\def\C{{\mathbb C}}
\def\N{{\mathbb N}}
\def\R{{\mathbb R}}
\def\Z{{\mathbb Z}}
\def\Q{{\mathbb Q}}
\def\A{{\mathbb A}}
\def\BA{{\mathbb A}}
\def\BR{{\mathbb R}}
\def\BC{{\mathbb C}}

\def\back{{\backslash}}
\newcommand{\calF}{\mathcal{F}}
\newcommand{\calP}{\mathcal{P}}
\newcommand{\calV}{\mathcal{V}}
\newcommand{\ago}{\mathfrak{a}}
\newcommand{\hDelta}{\widehat \Delta}
\newcommand{\htau}{\widehat \tau}
\newcommand{\sbs}{\subset}

\newcommand{\smin}{\smallsetminus}
\newcommand{\al}{\alpha}
\newcommand{\la}{\lambda}
\newcommand{\La}{\Lambda}
\newcommand{\bsl}{\backslash}
\newcommand{\calA}{\mathcal{A}}
\newcommand{\Rel}{\mathrm{Re}}

\newtheorem{thm}{Theorem}[section]
\newtheorem{cor}[thm]{Corollary}
\newtheorem{lem}[thm]{Lemma}

\newtheorem{prop}[thm]{Proposition}

\newtheorem*{conj*}{Conjecture}

\newtheorem{lemma}[thm]{Lemma}
\newtheorem{proposition}[thm]{Proposition}
\newtheorem{corollary}[thm]{Corollary}

\newtheorem{definition}[thm]{Definition}

\newtheorem{rmk}[thm]{Remark}
\newtheorem{remark}[thm]{Remark}

\newcommand{\mm}[4]{\left(\begin{smallmatrix} #1 & #2\\ #3 & #4\end{smallmatrix}\right)}

\DeclareMathOperator{\tr}{tr}

\DeclareMathOperator{\SO}{SO}
\DeclareMathOperator{\Spin}{Spin}

\DeclareMathOperator{\SL}{SL}
\DeclareMathOperator{\GL}{GL}

\DeclareMathOperator{\G_2}{G_2}
\DeclareMathOperator{\F_4}{F_4}
\DeclareMathOperator{\E_6}{E_6}

\DeclareMathOperator{\diag}{diag}

\DeclareMathOperator{\Ann}{Ann}

\def\NullU{{\Omega}}

\AtBeginDocument{%
   \def\MR#1{}
}

\begin{document}
\renewcommand{\theequation}{\arabic{equation}}
\numberwithin{equation}{section}

\title{A $\G_2$-period of a Fourier coefficient of an Eisenstein series on $\E_6$}
\author{Aaron Pollack}
\address{Department of Mathematics\\ Institute for Advanced Study\\ Princeton, NJ USA}\email{apollack@math.ias.edu}

\author{Chen Wan}
\address{Department of Mathematics\\ Institute for Advanced Study\\ Princeton, NJ USA}\email{wanxx123@math.ias.edu}

\author{Micha\l{} Zydor}
\address{Department of Mathematics\\ Institute for Advanced Study\\ Princeton, NJ USA}\email{mzydor@math.ias.edu}

\begin{abstract} We calculate a $\G_2$-period of a Fourier coefficient of a cuspidal Eisenstein series on the split simply-connected group $\E_6$, and relate this period to the Ginzburg-Rallis period of cusp forms on $\GL_6$. This gives us a relation between the Ginzburg-Rallis period and the central value of the exterior cube L-function of $\GL_6$.\end{abstract}
\maketitle
%\tableofcontents
%\setcounter{tocdepth}{1}

%=================================================
\section*{Introduction}

Let $F$ be a number field and $\A$ its ring of adeles.
Let $\pi$ be a cuspidal automorphic representation of $\GL_6(\BA)$ with trivial central character. For $\varphi \in \pi$, we define the Ginzburg-Rallis period as
$$
\mathcal{P}_{GR}(\varphi) =
\int\limits_{\GL_2(F) \bsl \GL_2(\A)^1} \int\limits_{(M_{2}(F) \bsl M_{2}(\A))^3} \varphi\left( \begin{pmatrix}I_2&x&z\\ 0&I_2&y\\ 0&0&I_2\end{pmatrix} \begin{pmatrix}h&0&0\\ 0&h&0\\ 0&0&h\end{pmatrix} \right)
 \psi(\tr(x)+\tr(y))\, dxdydz dh
$$
where $\GL_2(\A)^1$ is the subgroup of elements with norm of determinant $1$,
$M_{2}$ is the space of $2 \times 2$ matrices and $\psi : F \bsl \A \to \C^*$ is a non-trivial character.

In fact, one can also define the quaternion version of the Ginzburg-Rallis period for cusp forms on $\GL_3(D)(\BA)$ where $D/F$ is a quaternion algebra.
Let us denote this period by $\calP^{D}_{GR}$.
In \cite{ginRal}, Ginzburg-Rallis made the following conjecture about the relation between this period and the central value $L\left(\frac{1}{2}, \pi, \La^3\right)$ of the exterior cube L-function $L(s,\pi,\La^3)$.
\begin{conj*}[Ginzburg-Rallis] Let $\pi$ be a cuspidal automorphic representation of $\GL_6(\A)$ with trivial central character.
 The following are equivalent:
\begin{enumerate}
\item $L(\frac{1}{2},\pi,\La^3)\neq 0$.
\item There exists a unique quaternion algebra $D$ over $F$ (which may be split) and a cuspidal automorphic representation $\pi_D$
of $\GL_3(D)$, equivalent to $\pi$ at almost all places of $F$ (i.e., the global Jacquet-Langlands correspondence of $\pi$ from $\GL_6(\BA)$ to $\GL_3(D)(\BA)$), such that $\calP^{D}_{GR}$
does not vanish on the space of $\pi_D$.
\end{enumerate}
\end{conj*}

The purpose of this paper is to prove the following theorem which confirms
one direction of the Ginzburg-Rallis conjecture in the split case.
\begin{thm}\label{main 2}
Let $\pi$ be a cuspidal automorphic representation of $\GL_6(\BA)$ with trivial central character. Assume
that the exterior cube L-function $L(s,\pi,\La^3)$ is nonzero at $s=\frac{3}{2}$; this is always the case if $\pi$ is everywhere tempered. If there exists $\varphi \in \pi$ with $\mathcal{P}_{GR}(\varphi)\neq 0$
then $L(\frac{1}{2},\pi,\La^3)\neq 0$.
\end{thm}

\subsection{The strategy of the proof}

Our proof relies on the fact that the exterior cube $L$-function for $\GL_6$
is a Langlands-Shahidi $L$-function for the split, simply connected group of type $\E_6$. The idea is to relate the period $\calP_{GR}$
to a suitable period of a residual representation on $\E_6$ whose existence relies on the non-vanishing
of $L(1/2, \pi, \La^{3})$. This idea has been used in multiple works of Ginzburg, Jiang, Rallis
and Soudry, see \cite{jiang, gjr1, gjr2, gjr3} to cite a few.
The reference \cite{gjs} contains a comprehensive account of the method and results that can and have been obtained through it.
See also \cite{ichYam, ginLap}
for different approaches.

Let us describe the method in more detail.
Let $Q$ be the parabolic subgroup of the split, simply connected $\E_6$ whose Levi subgroup is
of $D_4$ type. In section \ref{ssec:periodDef} we define a certain generic character $\xi$
of the unipotent radical $N$ of $Q$ and we define a subgroup $H$ of $Q$ to which the character extends.
This group turns out to be isomorphic to $N \rtimes \G_2$, where $\G_2$ is the split exceptional group of type $\G_2$.
We can thus define the period
\[
\calP_{\G_2}(f) = \int_{H(F) \bsl H(\A)}f(h)\xi(h)\,dh
 \]
where $f$ is an automorphic form on $\E_6(\A)$. Let now $P$ denote the maximal parabolic subgroup of
$\E_6$ with Levi factor of type $A_5$. Given a cuspidal representation $\pi$ of $GL_6(\A)$
we can consider the Eisenstein series $E(\phi, s)$ which realizes the induction from
$\pi \otimes |\det |^{s}$
to $\E_6(\A)$ (see \ref{ssec:eis} for unexplained notation).
By the Langlands-Shahidi theory, the non-vanishing of the residue at $s=1/2$ of $E(\phi, s)$ implies that $L(\frac{1}{2},\pi,\La^3)\neq 0$.
The idea is to study $\calP_{\G_2}(Res_{s=1/2}E(\phi, s))$ and relate it to the Ginzburg-Rallis period. In fact,
we essentially show the two are equal.

More specifically, one starts by computing the orbits $P (F) \bsl \E_6(F) / H(F)$. This problem is challenging
when working with exceptional groups and we do this using the explicit realization of $\E_6$
as determinant preserving linear transformations of the $27$-dimensional exceptional Jordan algebra.
These computations allow us to compute $\calP_{\G_2}(\La^{T}E(\phi, s))$ explicitly. Here $\La^T$ is the Arthur-Langlands
truncation operator \cite{arthur}. The truncation operator brings us to the next difficulty in this approach, namely convergence issues. In fact, these can be quite daunting, especially when dealing with non-reductive periods. Indeed, unless we're dealing with compact periods, then even for cusp forms, convergence of 
periods requires a non-trivial argument, see for example \cite{jacShal}. In \cite{ichYam}, Ichino and Yamana handle this using a truncation procedure adapted
to their period. However, their case is reductive. Here, we do not introduce any new truncation. Instead,
we extend the analysis introduced in the appendix of \cite{B17} which is based on norms on adelic points of linear algebraic groups
and their quotients. In fact, as noticed in \emph{loc. cit.}, the key property is for the quotient to be quasi-affine. We believe this approach
should generalize to other periods.

The computation of $\calP_{\G_2}(\La^{T}E(\phi, s))$ allows us to easily deduce the following theorem.

\begin{thm}\label{main}
Let $\pi$ be a cuspidal automorphic representation of $\GL_6(\BA)$ with trivial central character.
Suppose that the Ginzburg-Rallis period $\mathcal{P}_{GR}(\varphi)$ is non-zero on the space of $\pi$.
Then, there exists $\phi$ in the induced space of $\pi$ such that $Res_{s=\frac{1}{2}} E(\phi,s) \neq 0$.
\end{thm}

\begin{corollary}
Theorem \ref{main} implies Theorem \ref{main 2}.
\end{corollary}

\begin{proof}
By Theorem \ref{main} the intertwining operator associated to $P$ has a pole at $s=1/2$.
By 2.5.3 in \cite{kim}, the normalizing factor of the intertwining operator is
\[
\dfrac{L(s, \pi, \La^3 )\zeta_F(2s)}{L(s+1, \pi, \La^3)\zeta_F(2s+1)}
\]
where $\zeta_F(s)$ is the Dedekind zeta function. By Theorem 4.11 of \cite{kim}, the normalized intertwining operator is holomorphic at $s=1/2$.
Since we have assumed that $L(3/2,\pi,\La^3)\neq 0$ it follows that the numerator $L(s, \pi, \La^3)\zeta_F(2s)$ has a pole at $s=1/2$, which implies that $L(\frac{1}{2},\pi,\La^3)\neq 0$.
\end{proof}

\begin{remark}
The $L$-function $L(s, \pi, \La^{3} )$ is the completed $L$-function as defined by Shahidi \cite{shah}.
Ginzburg-Rallis \cite{ginRal} prove that the partial $L$-function $L(s, \pi, \La^{3} )$ is holomorphic at $s=1/2$.
Unless $\pi$ is assumed everywhere tempered, it is unknown if the completed $L$-function is holomorphic at $s=1/2$.
The non-vanishing statement makes sense in either case nonetheless.
\end{remark}

\subsection{Structure of the paper}

In section \ref{sec:exceptional} we introduce the groups we are working with together with their natural realizations. The following section computes the orbits of $H$ on a flag variety of $\E_6$ associated to the parabolic subgroup of type $A_5$. Then in section \ref{sec:Eis} we set up notations relating to Eisenstein series and truncation. In section \ref{sec:PG2} we prove the main theorem (i.e. Theorem \ref{main}) by assuming the absolute convergence of certain integrals (i.e. Proposition \ref{prop:absConv}). Finally in section \ref{sec:convergence} we prove Proposition \ref{prop:absConv}.

\subsection{Acknowledgments} Throughout the course of this work, we have benefited from the warm hospitality of the Institute of Advanced Study.  We thank the IAS for providing an atmosphere conducive to research, and for making this collaboration possible.  While at the IAS, A.P. has been supported by the Schmidt fund at the IAS, C.W. has been supported by the NSF at the IAS, and M.Z. has been supported by the Bell companies fellowship fund and the NSF at the IAS.

\section{Preliminaries on exceptional groups and their embeddings}\label{sec:exceptional}  In this section we define the various reductive groups and parabolic subgroups involved in our period calculation.  We begin with the definitions of the octonions $\Theta$ and $\G_2$.  We then discuss the exceptional cubic norm structure $J = H_3(\Theta)$ and the group $\E_6$.  After defining $\E_6$, we discuss parabolic subgroups of $\G_2$ and $\E_6$.  We then define the formal period that we study.  Throughout this section, $F$ denotes an arbitrary ground field of characteristic $0$.

\subsection{The octonions and $\G_2$} In this subsection we recall the split octonions $\Theta$ and $\G_2$.  The reader may see \cite{springerVeldkampBook} for more background information pertinent to this subsection.  We first recall the Zorn model of the octonions, and then the Cayley-Dickson construction.  The group $\G_2$ is defined as the automorphisms of $\Theta$.

\subsubsection{Zorn model}
We recall the Zorn model of the octonions $\Theta$ over $F$.  Denote by $V_3$ the standard three-dimensional representation of $\SL_3$, and $V_3^\vee$ the dual representation.  So, $V_3$ comes equipped with an isomorphism $\wedge^3 V_3 \rightarrow F$
where $F$ is the ground field.  We fix a standard basis $e_1, e_2, e_3$ in $V_3$, so that $e_1 \wedge e_2 \wedge e_3 \mapsto 1$ under the identification $\wedge^3 V_3 \simeq F$, and denote by $e_1^*, e_2^*, e^*_3$ the dual basis in $V_3^\vee$.  The indentification $\wedge^3 V_3 \simeq F$ induces an isomorphism $\wedge^2 V_3 \simeq V_3^\vee$ and $\wedge^2 V_3^\vee \simeq V_3$.  Under these identifications, $e_1 \wedge e_2 \mapsto e_3^*$, $e_2 \wedge e_3 \mapsto e_1^*$, $e_1^* \wedge e_2^* \mapsto e_3$, etc.

In the Zorn model, an octonion $x$ is represented by a $2\times 2$ matrix $x = \mm{a}{v}{\phi}{d}$ where $a,d \in F$, $v \in V_3$ and $\phi \in V_3^\vee$.  One sets $x^* = \mm{d}{-v}{-\phi}{a}$ and $n(x) = ad- \phi(v)$ the conjugate and norm of $x$, respectively.  The trace of $x$, $\tr(x)$, is defined to be $\tr(x) = a+d$.  One defines the symmetric bilinear form $(\;,\;)$ on $\Theta$ by $(x,y) = n(x+y) - n(x) - n(y)$.  This form is non-degenerate.

The multiplication on $\Theta$ is given by the formula
\[x \cdot x' = \left(\begin{array}{cc} a & v \\ \phi & d\end{array}\right) \cdot \left(\begin{array}{cc} a' & v' \\ \phi' & d' \end{array}\right) = \left(\begin{array}{cc} aa'+ \phi'(v) & av'+d'v -\phi\wedge \phi' \\ a'\phi + d\phi'+ v\wedge v' & \phi(v')+dd' \end{array}\right).\]
This multiplication is neither commutative nor associative, but satisfies the following three important identities:
\begin{enumerate}
\item $n(x \cdot x') = n(x) n(x')$;
\item $\tr(x_1 (x_2 x_3)) = \tr((x_1 x_2) x_3)$;
\item $x^* (x y) = (x^* x) y = n(x) y$.
\end{enumerate}
In fact, the subalgebra generated by any two elements $x, y$ of $\Theta$ is associative.

For later reference, we denote $\epsilon_1 = \begin{pmatrix}1&0\\0&0\end{pmatrix}$ and $\epsilon_2 = \begin{pmatrix}0&0\\0&1\end{pmatrix}$ in $\Theta$, so that $1 = \epsilon_1 + \epsilon_2$.  We sometimes abuse notation and write $e_j$ or $e_k^*$ for the corresponding element $\mm{0}{e_j}{0}{0}$ or $\mm{0}{0}{e_k^*}{0}$ of $\Theta$.

\subsubsection{Cayley-Dickson construction}\label{subsubsec:CD} One can also construct the octonions by ``doubling" a quaternion algebra.  This is called the Cayley-Dickson construction.

In general, if $D$ is a quaternion algebra and $\gamma \in \GL_1(F)$, then $\Theta_{D,\gamma} = D \oplus D$ is an Octonion algebra, with addition defined component-wise and multiplication given by
\[(x_1,y_1) \cdot (x_2,y_2) = (x_1 x_2 + \gamma y_2^* y_1, y_2 x_1 + y_1 x_2^*).\]
The conjugate of $(x,y)$ is $(x^*,-y)$ and the trace of $(x,y)$ is $\tr_{D}(x)$.  The norm of $(x,y)$ is $n_D(x) -\gamma n_D(y)$.  Here $n_D$ and $\tr_D$ are the reduced norm and trace on $D$.

The octonions $\Theta_{D,\gamma}$ are split precisely when $\gamma$ is a (reduced) norm from $D^\times$. In particular, if either $D = M_2(F)$ is split, or if $\gamma = 1$, then $\Theta_{D,\gamma}$ is split.

\subsubsection{The group $\G_2$} The linear algebraic group $\G_2$ is by definition the automorphisms of the octonion algebra $\Theta$.  That is, if $g \in \G_2(F)$, then $g1 = 1$, and for all $x,y \in \Theta$, $g (x^*) = (gx)^*$, $(gx) \cdot (gy) = g (x \cdot y)$.  In fact, it can be shown that the first two statements follow from the last one.

\subsection{The exceptional cubic norm structure $J$ and the group $\E_6$}\label{ssec:E6}  In this subsection we recall the definition of the exceptional cubic norm structure $J = H_3(\Theta)$ and the split simply-connected group $\E_6$.  We will also recall some facts about the Lie algebra of $\E_6$, which we will need later.

\subsubsection{The exceptional cubic norm structure $J$} Define $J = H_3(\Theta)$ to be the $F$ vector space of elements
\[X = \left(\begin{array}{ccc} c_1& a_3 & a_2^* \\a_3^* &c_2&a_1\\a_2&a_1^*&c_3\end{array}\right)\]
where $c_1, c_2, c_3 \in F$ and $a_1, a_2, a_3 \in \Theta$.  Thus, $J$ is $27$-dimensional.

The space $J$ comes equipped with a cubic norm map $n: J \rightarrow F$ and a quadratic adjoint map $\#: J \rightarrow J$.  The norm of $X$, $n(X)$, is defined to be
\begin{equation}\label{eqn:nX}n(X) = c_1 c_2 c_3 - c_1 n(a_1) - c_2 n(a_2) - c_3n(a_3) + \tr(a_1a_2a_3).\end{equation}
Polarizing the norm form, one obtains a symmetric trilinear form $J \otimes J \otimes J \rightarrow F$, normalized by the identity $(X,X,X) = 6n(X)$.

This symmetric trilinear form gives rise to a bilinear map, denoted $\times$, from $J \otimes J$ to $J^\vee$.  Namely, if $x,y, z \in J$, then $x\times y \in J^\vee$ is by definition the linear form given by $(x\times y)(z) = (x,y,z)$.  For $x \in J$, one sets $x^\# = \frac{1}{2} x \times x$. If one puts on $J$ the pairing $(\;,\;): J \otimes J\rightarrow F$ given by
\begin{equation}\label{Jpairing}(X,X') = c_1c_1'+c_2c_2'+c_3c_3' + (a_1,a_1') +(a_2,a_2') + (a_3,a_3'),\end{equation}
then this pairing induces an identification $J^\vee \simeq J$.  Under this identification, the quadratic map $\#: J \rightarrow J$ and the bilinear map $\times$ can be written simply in the coordinates of $J$.  Namely, one obtains
\[X^\# = \left(\begin{array}{ccc} c_2c_3 - n(a_1) & a_2^*a_1^*-c_3 a_3 & a_3a_1 - c_2 a_2^* \\ a_1 a_2 - c_3 a_3^* &c_3 c_1 - n(a_2) & a_3^*a_2^*-c_1a_1\\ a_1^* a_3^* - c_2 a_2 & a_2 a_3 - c_1 a_1^* &c_1 c_2 - n(a_3)\end{array}\right)\]
and $X \times Y = (X+Y)^\# - X^\# - Y^\#$.

\subsubsection{The algebraic group $\E_6$} The linear algebraic group $\E_6$ is by definition the linear automorphisms of $J$ that preserve the norm form.  We let $\E_6$ act on the right of $J$.  That is,
\[\E_6 = \{g \in \GL(J): n(X g) = n(X) \text{ for all } X \in J\}.\]
Equivalently, $\E_6$ is the subgroup of $\GL(J)$ fixing the symmetric trilinear form on $J$:
\[\E_6 = \{g \in \GL(J): (xg,yg,zg) = (x,y,z) \forall x,y,z \in J\}.\]
Note that $\E_6$ \emph{does not} fix the pairing \eqref{Jpairing}; the subgroup of $\E_6$ that also fixes this pairing is $\F_4$.  If $g \in \GL(J)$, denote by $\widetilde{g} \in \GL(J)$ the unique map satisfying $(Xg,Y) = (X,Y\widetilde{g})$ for all $X,Y$.  If $g \in \E_6$, one has $(Xg) \times (Yg) = (X \times Y) \widetilde{g}^{-1}$ for all $X, Y \in J$.  In other words, $g \mapsto \widetilde{g}^{-1}$ is an involution on $\E_6$, whose fixed point set is $\F_4$.

\subsubsection{The Lie algebra of $\E_6$} Below we will need a few facts about the Lie algebra of $\E_6$, specifically some nilpotent elements of it.  The Lie algebra of $\E_6$ consists of the elements $\phi \in End(J)$ satisfying
\[(\phi(x),y,z) + (x,\phi(y),z) + (x,y,\phi(z)) = 0\]
for all $x,y,z \in J$.
If $\gamma \in J^\vee$, and $v \in J$, define $\Phi_{\gamma,v} \in End(J)$ as
\[\Phi_{\gamma,v}(z) = -\gamma \times (v \times z) + (\gamma,z)v + (\gamma,v)z\]
and $\Phi'_{\gamma,v} = \Phi_{\gamma,v} - \frac{2}{3}(\gamma,v)$.
One has the following well-known proposition.
\begin{proposition}Suppose $\gamma \in J^\vee$ and $v \in J$. We have the following facts.
\begin{enumerate}
\item For $\gamma \in J^\vee$ and $v \in J$, one has
\[(\Phi_{\gamma,v}(x),y,z) + (x,\Phi_{\gamma,v}(y),z) + (x,y,\Phi_{\gamma,v}(z)) = 2(v,\gamma) (x,y,z)\]
for all $x,y,z \in J$.  Consequently, $\Phi'_{\gamma,v} \in Lie(\E_6)$.
\item If $\gamma^\# = 0$ and $(\gamma,v) = 0$, then $\Phi'_{\gamma,v} = \Phi_{\gamma,v}$ satisfies $\Phi_{\gamma,v}^2(z) = -2(\gamma,z) \gamma \times v^\#$ and $\Phi_{\gamma,v}^3 = 0$. Consequently, $\exp(\Phi_{\gamma,v})$ is a unipotent element of $\E_6$.
\item Similarly, if $v^\# = 0$ and $(\gamma,v) = 0$, then $\Phi_{\gamma,v} = \Phi'_{\gamma,v}$ satisfies $\Phi_{\gamma,v}^2(z) = -2(\gamma^\# \times v,z)v$ and $\Phi_{\gamma,v}^3 = 0$.  Consequently, $\exp(\Phi_{\gamma,v})$ is a unipotent element of $\E_6$.
\end{enumerate}\end{proposition}
\begin{proof} The first item is essentially \cite[Equation (9)]{rumelhart}.  The second and third items are essentially contained in \cite[Lemma 1 and Proposition 5]{springer}.\end{proof}

\subsection{Parabolic subgroups of $\G_2$}  We will need some facts about parabolic subgroups of $\G_2$ and $\E_6$. We begin with parabolic subgroups of $\G_2$. Let $V_7 \subseteq \Theta$ be the subspace of traceless elements of $\Theta$, i.e.,
\[V_7 = \{x \in \Theta: \tr(x) = 0\}.\]
Equivalently, $V_7$ is the perpendicular space to $1 \in \Theta$.

The parabolic subgroups of $\G_2$ can be defined as the stabilizers of certain subspaces or flags in $V_7$.  To setup the statement of this fact, we require a definition pertaining to two-dimensional isotropic subspaces of $V_7$.

Thus, suppose $\NullU \subseteq \Theta$ is a two dimensional isotropic subspace.  Note that if $x,y$ is a basis of $\NullU$, then the line spanned by $x^*\cdot y$ is independent of the choice of basis.  Similarly, the line spanned by $x \cdot y^*$ is independent of the choice of basis.
\begin{definition} Suppose that $\NullU \subseteq \Theta$ is a two-dimensional isotropic subspace, and $x, y$ is a basis of $\NullU$.  Call $\NullU$ \emph{left null}  if $x^* \cdot y = 0$ for a basis $x,y$ of $\NullU$. Similarly, call $\NullU$ \emph{right null} if $x \cdot y^*  = 0$ for a basis $x,y$ of $\NullU$.  If $\NullU \subseteq V_7$, then $\NullU$ is left-null if and only if it is right-null.  We say $\NullU$ is \emph{null} (with no-modifier) if it is both left and right null.\end{definition}
For general isotropic $\NullU$, it could be left-null but not right-null, right-null but not left-null, both, or neither. We now have the following well-known proposition.

\begin{proposition}\label{prop:G2parab} The group $\G_2$ has three conjugacy classes of parabolic subgroups: two maximal ones and a Borel.  These parabolic subgroups and their flag varieties are characterized as follows.
\begin{enumerate}
\item The group $\G_2$ acts transitively on the set of isotropic lines $\ell \subseteq V_7$.  The stabilizer $P(\ell)$ of such a line is a maximal parabolic subgroup of $\G_2$, with reductive quotient isomorphic to $\GL_2$, and whose unipotent radical $N(\ell)$ is a $3$-step unipotent group.
\item The group $\G_2$ acts transitively on the set of null isotropic two spaces $\NullU \subseteq V_7$.  The stabilizer $P(\NullU)$ of such a two-space is a maximal parabolic subgroup of $\G_2$, with reductive quotient isomorphic to $\GL_2$, and whose unipotent radical $N(\NullU)$ is a $2$-step unipotent group.
\item The group $\G_2$ acts transitively on the set of pairs $(\ell,\NullU)$ with $\ell \subseteq \NullU \subseteq V_7$ an isotropic line contained in a null isotropic two-space of $V_7$.  The stabilizer $B(\ell,\NullU)$ of such a pair is a Borel subgroup of $\G_2$.\end{enumerate}\end{proposition}

\subsection{Singular subspaces of $J$ and parabolics of $\E_6$}
We now recall facts pertaining to two of the conjugacy classes of parabolic subgroups of $\E_6$ that we will compute with below.
For more on parabolics of $\E_6$ and their flag varieties, see \cite{garibaldi}.
If $V \subseteq J$ is a subspace, following \cite{garibaldi} one says that $V$ is \emph{totally singular} if $v^\# = 0$ for all $v \in V$, or equivalently, if $v_1 \times v_2 = 0$ for all $v_1, v_2 \in V$.

\subsubsection{The $A_5$ parabolics}\label{sssec:A5parab} We now describe the flag variety of the maximal parabolics $P \subseteq \E_6$ whose reductive quotient are of Dynkin type $A_5$.

\begin{proposition}[\cite{springerVeldkampHM},\cite{garibaldi}]\label{prop:A5parab} One has the following facts regarding the conjugacy class of maximal parabolic subgroups of $\E_6$ of type $A_5$.
\begin{enumerate}
\item The group $\E_6$ acts transitively on the totally singular $6$-dimensional subspaces $V$ of $J$, and the stabilizer $P(V)$ of such a $V$ is a maximal parabolic subgroup of $\E_6$.
\item The reductive quotient $M(V)$ of $P(V)$ is isomorphic to the group $\{(\lambda, g) \in \GL_1 \times \GL(V): \lambda^3 = \det(g)\}$.
\end{enumerate}\end{proposition}
\begin{proof} The first item is \cite[Theorem 7.2]{garibaldi}, for which \emph{loc. cit.} references \cite[Proposition 3.14]{springerVeldkampHM}.  The second item follows from, for instance, \cite[2.5.3]{kim}.\end{proof}
For example, suppose $\NullU \subseteq V_7 \subseteq \Theta$ is a null-isotropic two space.  Then
\begin{equation}\label{eqn:V(U)}V(\NullU) = \left\{ x = \left(\begin{array}{ccc} 0& u_3 & u_2^* \\ u_3^* & 0 & u_1 \\ u_2 & u_1^* & 0 \end{array}\right): u_1, u_2, u_3 \in \NullU \right\}\end{equation}
is a six-dimensional totally singular subspace of $J$.

\subsubsection{The $D_4$ parabolics}\label{subsec:D4parab} We now describe the flag variety of the (non-maximal) parabolic subgroups $Q \subseteq \E_6$ whose reductive quotients are of type $D_4$.  If $\ell \subseteq J$ is a one-dimensional totally singular subspace, and $\ell' \in J^\vee$ is a one-dimensional totally singular subspace, one says that $\ell$ and $\ell'$ are \emph{incident} if $\Phi'_{\gamma, v } = 0$ for $v \in \ell$ and $\gamma \in \ell'$.
\begin{proposition}[\cite{springerVeldkampHM},\cite{garibaldi}] The group $\E_6$ acts transitively on pairs $(\ell,\ell')$ of incident one-dimensional singular subspaces of $J$, $J^\vee$.  The stabilizer $Q(\ell,\ell')$ is a parabolic subgroup of $\E_6$ whose reductive quotient is of type $D_4$.\end{proposition}
\begin{proof} This proposition is contained in \cite[Theorem 7.2]{garibaldi}.  See also \cite[Section 3]{springerVeldkampHM}. To compare with \cite{garibaldi}, the singular line $\ell' \subseteq J^\vee$ gives a subspace $\ell' \times J^\vee \subseteq J$ which is a ``$6$-space'', in the parlance of \emph{loc. cit.}.\end{proof}

For $i =1,2,3$, denote $e_{ii}$ the element of $J$ with a $1$ in the $(i,i)$-th place.  Then, if $\gamma = e_{11}$ is considered as an element of $J^\vee$ via the pairing $(\,,\,)$ on $J$, and $v = e_{33} \in J$, then the lines $F e_{11}$ and $F e_{33}$ are incident.  Denote by $Q$ the parabolic subgroup that stabilizes these lines.  Defined in terms of flags of $J$, $Q$ is the subgroup of $\E_6$ that fixes the chain of subspaces
\[V_1=\left(\begin{array}{ccc} 0 & 0&0\\ 0&0&0\\0&0&*\end{array}\right) \subseteq \left(\begin{array}{ccc} 0 & 0&0\\ 0&*&*\\0&*&*\end{array}\right) = V_{10}.\]

Below we will use the fact that $Q$ stabilizes a larger flag of $J$.  We now describe this flag.  We begin with a piece of notation.  Suppose $V$ is a subspace of $J$, and $I$ is a subset of the indices $\{c_1,c_2,c_3,a_1,a_2,a_3\}$.  We define $V(I)$ to be the subspace of $V$ consisting of elements $x$ so that all the nonzero coordinate entries of $x$ are labelled by indices in $I$.  Put another way, $c_1, c_2, c_3$ define linear functionals $J \rightarrow F$, and $a_1, a_2, a_3$ define linear functions $J \rightarrow \Theta$.  Then
\[V(I) := \{x \in V: \alpha(x) = 0 \text{ for all } \alpha \notin I\}.\]

With this notation, we define two increasing filtrations on $J$:
\[\mathcal{F}^{1}_{\bullet}: 0 \subseteq J(c_3) \subseteq J(c_3,a_1) \subseteq J(c_3,a_1,a_2) \subseteq J(c_3, a_1,a_2, c_2) \subseteq J(c_3,a_1,a_2,c_2,a_3) \subseteq J\]
and
\[\mathcal{F}^{2}_{\bullet}: 0 \subseteq J(c_3) \subseteq J(c_3,a_1) \subseteq J(c_3,a_1,c_2) \subseteq J(c_3, a_1,c_2, a_2) \subseteq J(c_3,a_1,c_2,a_2,a_3) \subseteq J.\]
That is, define $\mathcal{F}^{1}_{0}(J) = 0$, $\mathcal{F}^{1}_{1}(J) = J(c_3)$, $\mathcal{F}^{1}_{2}(J) = J(c_3, a_1)$ etc, and similarly for $\mathcal{F}^{2}_i(J)$.  For a subspace $V$ of $J$, set $\mathcal{F}^{?}_{i}(V) = \mathcal{F}^{?}_{i}(J) \cap V$ for $? = 1,2$.

The parabolic $Q$ fixes both of these flags on $J$.  To ses this, note that since $Q$ fixes $V_1$, it fixes
\[V_{17}=\left(\begin{array}{ccc} 0 & 0&* \\ 0&0&*\\ *&*&*\end{array}\right)\]
since $V_{17}$ is the set of elements $X \in J$ with $X \times V_{1} = 0$.  Since $Q$ fixes $V_{10}$, it fixes
\[V_{26} = \left(\begin{array}{ccc} 0 & *&*\\ *&*&*\\ *&*&*\end{array}\right)\]
since $V_{26}$ consists of the $X \in J$ such that $(V_{10},V_{10},X) = 0.$  One finds that $Q$ fixes the flags above by taking intersections of the spaces $V_{ij}$ above.

As a Levi subgroup $L$ of $Q$, we take the subgroup of $g \in \E_6$ so that $g$ fixes each of the spaces $J(c_1)$, $J(c_2)$, $J(c_3)$, $J(a_1)$, $J(a_2)$ and $J(a_3)$.

\subsection{The unipotent radical of $Q$} In this subsection, we describe explicitly the unipotent radical $N$ of $Q$.  The material in this subsection is surely well-known.  However, we are unaware of a reference, so we include some sketches of the proofs.

First, denote by $N(\Theta)$ the ``Heisenberg'' group of $\Theta$.  By this we mean that elements of $N(\Theta)$ consist of triples $(x,y;z)$ with $x,y,z \in \Theta$ and multiplication is given by
\[(x,y;z)(x',y';z') = (x+x', y+y'; z+z' + xy').\]
Then the inverse of $(x,y;z)$ is
\[(x,y;z)^{-1} = (-x,-y;xy-z).\]

The unipotent radical $N$ is isomorphic with $N(\Theta)$, as we now explain.  First, we explain how $N(\Theta)$ acts (on the right) of $J$.  Namely,
\begin{equation}\label{eqn:Ntheta} \left(\begin{array}{ccc} c_1& a_3 & a_2^* \\a_3^* &c_2&a_1\\a_2&a_1^*&c_3\end{array}\right)\mapsto \left(\begin{array}{ccc} 1&  &  \\ x^*&1&\\z^*&y^*&1\end{array}\right) \left\{ \left(\begin{array}{ccc} c_1& a_3 & a_2^* \\a_3^* &c_2&a_1\\a_2&a_1^*&c_3\end{array}\right) \left(\begin{array}{ccc} 1& x & z \\ &1&y\\&&1\end{array}\right)\right\}.\end{equation}
Since the multiplication in $\Theta$ is not associative, it is not \emph{a priori} clear that the element above is in $H_3(\Theta)$, that this defines an action of $N(\Theta)$ on $J$, or that this action preserves the norm.  However, all these facts are true, as we now explain.

First, multiplying the second and third terms, and then multiplying with the first term, one finds by simple explicit computation that the resulting matrix is still Hermitian, i.e., is in $H_3(\Theta)$.  We denote by $n(x,y;z)$ the element of $\GL(J)$ prescribed by \eqref{eqn:Ntheta}.  To see that this matrix multiplication defines an action, we require the following lemma, which describes the matrix product \eqref{eqn:Ntheta} in coordinates.
\begin{lemma}\label{HeisTransform} Under the matrix product \eqref{eqn:Ntheta}, the $c_i$ and $a_i$ transform as follows:
\begin{itemize}
\item $c_1 \mapsto c_1$.
\item $c_2 \mapsto c_2 + (x,a_3) + c_1n(x)$.
\item $c_3 \mapsto c_3 + (a_1,y) + c_2n(y) + (a_2^*,z) + c_1n(z) + \tr(z^*a_3y)$.
\item $a_1 \mapsto a_1 + c_2y + a_3^*z + x^*a_2^* + x^* (a_3 y) + x^*c_1 z$.
\item $a_2 \mapsto a_2 + y^*a_3^* + z^*c_1.$
\item $a_3 \mapsto a_3 + c_1 x$.
\end{itemize}
\end{lemma}
\begin{proof} This is a direct computation. \end{proof}

Using this lemma, one can check that the map $N(\Theta) \rightarrow \GL(J)$ given by $(x,y;z) \mapsto n(x,y;z)$ defines a right-action of $N(\Theta)$ on $J$:
\begin{lemma} If $Z \in J$, then $(Z \cdot n(x,y;z)) \cdot n(x',y';z') = Z \cdot n(x+x',y+y';z+z' + xy')$.\end{lemma}
\begin{proof} This is almost entirely a simple direct check, using Lemma \ref{HeisTransform}.  For the reader checking this themselves, note that it is easy to see that $c_1, c_2, a_2$ and $a_3$ transform as required.  To see that $c_3$ transforms in the right way, one must use the identity $\tr((x_1 x_2) x_3) = \tr(x_1 (x_2 x_3)) = \tr(x_3 (x_1 x_2))$ for $x_1, x_2, x_3 \in \Theta$.  To see that $a_1$ transforms in the right way, one must use the identity $x^* (a_3 y') + a_3^* (xy') = (x,a_3) y'$, which comes from $(x^* + a_3^*) ((x+a_3) y') = n(x+ a_3)y'$ by linearizing. \end{proof}

We now explain why the $n(x,y;z)$ are in $\E_6$, and make up the unipotent radical $N$ of $Q$.  For $x,y, z \in \Theta$, write
\[ Y(x,y,z) = \left(\begin{array}{ccc} 0& x & z\\ x^* & 0 & y \\ z^* & y^* & 0 \end{array}\right) \in J.\]
Recall the elements $e_{11}, e_{33}$ defined above.
\begin{lemma}\label{lem:n0} One has $n(0,y;z) = \exp(\Phi'_{V(0,y,z),e_{33}})$ and $n(x,0;z) = \exp(\Phi'_{e_{11},V(x,0,z)})$.\end{lemma}
\begin{proof} One computes easily how the coordinates $a_i$ and $c_j$ transform under the map $\Phi'_{V(0,y,z),e_{33}}$, and the same for $\Phi'_{e_{11},V(x,0,z)}$.  That $n(0,y;z) = \exp(\Phi'_{V(0,y,z),e_{33}})$ again follows directly from this computation, by comparing with Lemma \ref{HeisTransform}, and similarly for $n(x,0;z)$.\end{proof}

We thus have the following corollary, which explicitly describes the unipotent radical $N$ of $Q$.
\begin{corollary} The elements $n(x,y;z)$ are in $N \subseteq \E_6$, and the map $N(\Theta) \rightarrow N$ is an isomorphism. \end{corollary}
\begin{proof} Since $n(x,y;z) = n(0,y;z)n(x,0,0)$, to check that the $n(x,y;z)$ are in $N$ it suffices to check that the $n(0,y;z)$ and $n(x,0;z)$ are in $N$.  By Lemma \ref{lem:n0}, these elements are exponentials of nilpotent elements of $Lie(\E_6)$. Hence it suffices to check the corresponding Lie algebra statement.  Now, one immediately has $\Phi'_{V(0,y,z),e_{33}}(e_{33}) = 0$.  Furthermore, if $\Phi^\vee_{\gamma,v}$ denotes how $\Phi_{\gamma,v}$ acts on the dual representation, then
\[\Phi^\vee_{\gamma, v}(\mu) = v \times (\gamma \times \mu) - (v,\mu)\gamma - (v,\gamma)\mu\]
for $\mu \in J^\vee$.  Using the fact that $e_{33}\times (V(x,y,z) \times e_{11}) = 0$, one sees that $\Phi^\vee_{V(0,y,z),e_{33}}(e_{11}) = 0$.  Thus, $\Phi'_{V(0,y,z),e_{33}}$ is in the Lie algebra of $N$.  One similarly finds that $\Phi'_{e_{11},V(x,0;z)}$ is in the Lie algebra of $N$.  Thus the map $N(\Theta) \rightarrow \E_6$ lands in $N$, as claimed.

To see that $N(\Theta) \rightarrow N$ is an isomorphism, note that it is clear from Lemma \ref{HeisTransform} that $N(\Theta) \rightarrow \E_6$ is injective.  But $N(\Theta)$ and $N$ are each unipotent groups of dimension $24$, so the isomorphism follows. \end{proof}

\subsection{The period}\label{ssec:periodDef}
In this subsection we define the period that we study in this paper.
Recall the Levi subgroup $L$ of $Q$, which acts on $J$ fixing the ``coordinate spaces'' $J(a_i)$, $J(c_i)$, $i = 1,2,3$.
We map $\G_2 \rightarrow L \subseteq \E_6$ by letting $\G_2$ act diagonally on the spaces $J(a_i)$ and trivially on the $J(c_i)$.  That is, if $g \in \G_2$ then
\[
\left(\begin{array}{ccc} c_1& a_3 & a_2^* \\a_3^* &c_2&a_1\\a_2&a_1^*&c_3\end{array}\right) \cdot g =
\left(\begin{array}{ccc} c_1& g^{-1}(a_3) & g^{-1}(a_2^*) \\ g^{-1}(a_3^*) &c_2&g^{-1}(a_1)\\g^{-1}(a_2)&g^{-1}(a_1^*)&c_3\end{array}\right).
\]
It is clear that this action preserves the norm on $J$, and thus $\G_2 \subseteq L \subseteq \E_6$.

Fix an additive character $\psi: F\backslash \BA \rightarrow \BC^\times$.
We put on $N \simeq N(\Theta)$ the character $\xi: N(F)\backslash N(\BA) \rightarrow \BC^\times$ defined by $\xi(n(x,y;z)) = \psi(\tr(x+y))$.
Note that if $g \in \G_2$, then
\begin{equation}\label{eqn:gconjN}g n(x,y;z) g^{-1} = n(gx, gy; gz).\end{equation}
It follows that $\G_2$ stabilizes this character on $N(\Theta)$.
In fact, one can show that $\G_2$ is the full stabilizer in $L$ of this character on $N$, but we do not need this fact, so we omit it.

If $\varphi$ is an automorphic function on $\E_6(F)\back \E_6(\BA)$, then the period we study is
\begin{equation}
\label{eqn:Period} \mathcal{P}_{\G_2}(\varphi)= \int_{\G_2(F)\backslash \G_2(\A)}\int_{N(F)\backslash N(\A)}{\xi(n)\varphi(nr)\,dn\,dr},
\end{equation}
assuming that the integral converges. To simply the notation, we will sometimes use $H=H_0N$ to denote the group $\G_2N$ with $H_0=\G_2$ and we will extend the character $\xi$ to $H$ by making it trivial on $H_0$.

\section{Orbits of $H$ on a flag variety of $\E_6$}\label{sec:orbits}
Let $\calV$ be the variety of six-dimensional totally singular subspaces of $J$.
The group $\E_6$ acts on the right on $\calV$ and Proposition \ref{prop:A5parab}
shows that this action is transitive with stabilizers being the $A_5$ parabolic subgroups of $\E_6$.
As explained in the introduction, we want to compute the period $\calP_{\G_2}$ of a truncated Eisenstein series
of $\E_6$ induced from an  $A_5$ parabolic subgroup. To this end, we study in this section the
orbits of $H=\G_2N$ on $\calV$ and we analyze the stabilizers of each orbit.

\subsection{The $Q$ orbits on $\calV$}\label{ssec:}

Since $Q$ preserves the filtration $\mathcal{F} = \mathcal{F}^{1}$, it is easy to write down an invariant of the orbits of $Q$ on $\calV$.  Namely, if $V$ is a six-dimensional totally singular subspace of $J$, then the dimensions $\mathcal{F}_i(V)/\mathcal{F}_{i-1}(V)$ are $Q$-invariants.  Here recall that $\mathcal{F}^{1}_i(V) = \mathcal{F}^{1}_i(J) \cap V$; see subsection \ref{subsec:D4parab}.  We will see that there are $7$ orbits, and that these dimensions completely characterize the orbits.

We will first produce representatives for these seven orbits, and then explain why they are only orbits.  Now, and below, we will use the following notation.  If $W$ is a subspace of $\Theta$, set
\[\Ann_{R}(W) = \{x \in \Theta: w \cdot x = 0 \text{ for all } w \in W\}\]
and similarly
\[\Ann_{L}(W) = \{x \in \Theta: x \cdot w = 0 \text{ for all } w \in W\}\]
the right and left annihilators of $W$. Note that the identity $x^* (x w) = n(x) w$ implies that if $W$ is nonzero, then $\Ann_{L}(W)$ is isotropic.  Similarly, if $W$ is nonzero, then $\Ann_{R}(W)$ is isotropic.  If $W \subseteq V_7$, then $w^* = -w$ for every $w \in W$.  It follows that $\Ann_{R}(W) = \Ann_{L}(W)^*$ when $W$ is traceless.  If $W$ is not contained in $V_7$, then $\Ann_{R}(W)$ need not be related to $\Ann_{L}(W)$.

Suppose $(V_1,V_2,V_3)$ in $\Theta^3$ satisfies $V_i \cdot V_{i+1} = 0$ (indices taken modulo $3$) and each $V_i$ is two-dimensional.  Furthermore, suppose that $\ell \subseteq \Theta$ is an isotropic line.  In this case, $\Ann_{R}(\ell) = \ell^* \Theta$, $\Ann_{L}(\ell) = \Theta \ell^*$, and both are four-dimensional. With this notation, the seven orbits of $Q$ on the totally singular six-dimensional subspaces of $J$ are represented as follows.

\begin{align}
\label{item222} V &= \left(\begin{array}{ccc} 0 & V_3 & * \\ * & 0 & V_1 \\ V_2 & * & 0 \end{array}\right)\\
\label{item411R1}V &= \left(\begin{array}{ccc} 0 & 0& * \\ 0 & 0& \ell \\ \Ann_{R}(\ell) & * & * \end{array}\right)\\
\label{item411R2} V &= \left(\begin{array}{ccc} * & \Ann_{R}(\ell) & * \\ * & 0 &0 \\ \ell & 0 & 0 \end{array}\right)\\
\label{item411R3} V &= \left(\begin{array}{ccc}  0 & \ell & 0 \\ * & * & \Ann_{R}(\ell) \\ 0 & * & 0 \end{array}\right)\\
\label{item411L1} V &= \left(\begin{array}{ccc}  0 & \Ann_{L}(\ell) & 0 \\ * & * & \ell \\ 0 & * & 0 \end{array}\right)\\
\label{item411L2} V &= \left(\begin{array}{ccc} 0 & 0&  *\\ 0 & 0 & \Ann_{L}(\ell) \\ \ell & * & * \end{array}\right)\\
\label{item411L3} V &= \left(\begin{array}{ccc} * & \ell & * \\ * & 0 &0 \\ \Ann_{L}(\ell) & 0 & 0 \end{array}\right)
\end{align}

It is clear by considering the dimensions $\mathcal{F}_{i}(V)/\mathcal{F}_{i-1}(V)$ that the above $V$ are in different $Q$-orbits.  We call spaces of the form \eqref{item222} a ``$(2,2,2)$'' orbit, spaces of the form \eqref{item411R1}, \eqref{item411R2}, and \eqref{item411R3} ``right $(4,1,1)$'' orbits, and spaces of the form \eqref{item411L1}, \eqref{item411L2}, and \eqref{item411L3} ``left $(4,1,1)$'' orbits.

We now show that these are the only orbits.  Note that the $V$ above are all direct sums of their intersections with the coordinate space $J(a_i)$, $J(c_i)$.  The Bruhat decomposition implies that every $Q$ orbit on $\calV$ has such a representative.

\begin{lemma}\label{lem:Bruhat}
Suppose $V$ is totally singular six-dimensional subspace of $J$.  Then there exists $q \in Q$ so that $V q$ is a direct sum of its intersections with with the coordinate spaces $J(a_i)$ and $J(c_i)$, $i = 1,2,3$.
\end{lemma}

\begin{proof}
Let $P$ be the stabilizer of $V$ in $\E_6$.
Denote by $T'$ the diagonal maximal torus of $\GL(V)$, and set $T = \{(\lambda, t) \in \GL_1 \times T': \lambda^3 = \det(t)\}$.
It follows from Proposition \ref{prop:A5parab} that $T$ is a maximal torus of a Levi subgroup of $P$, and thus a maximal torus of $\E_6$.  By the Bruhat decomposition, there exists $q \in Q$ and $w \in N(T)$, the normalizer of $T$, so that $V q =  V(\NullU) w := V'.$  Here $V(\NullU)$ was the singular space used to define $P$. Suppose $t \in T$.  Since $w$ normalizes $T$ and $T$ fixes $V(\NullU)$, $V' t = V'$.  It follows that $V'$ is a direct sum of its $T$-eigenspaces.  But the subtorus $T_0 \subseteq T$ consisting of the $(1, \diag(t_1,t_1,t_2,t_2,t_3,t_3))$ with $t_1 t_2 t_3 = 1$ has the coordinate spaces $J(a_i)$, $J(c_i)$ as its eigenspaces.  The lemma follows.
\end{proof}

Using Lemma \ref{lem:Bruhat}, the fact that the seven spaces $V$ above represent all orbits of $Q$ on $\calV$ is now an easy exercise using the following two well-known lemmas.
\begin{lemma} Suppose $V_1, V_2 \subseteq \Theta$ are nonzero subspaces, and $V_1 \cdot V_2 = 0$.  Then both $V_i$ are isotropic, and (at least) one of the following three things are true:
\begin{enumerate}
\item $V_1$ is one-dimensional, and $V_2 \subseteq \Ann_{R}(V_1)$, which is four-dimensional.  In particular, $\dim(V_1) + \dim(V_2) \leq 5$.
\item $V_2$ is one-dimensional, and $V_1 \subseteq \Ann_{L}(V_2)$, which is four-dimensional.  In particular, $\dim(V_1) + \dim(V_2) \leq 5$.
\item Neither $V_1$ nor $V_2$ is one-dimensional.  In this case, $V_1$ is two-dimensional, $V_2 = \Ann_{R}(V_1)$ is two-dimensional, and there is a unique two-dimensional subspace $V_3 \subseteq \Theta$ so that $V_i \cdot V_{i+1} = 0$ for $i = 1,2,3$ (indices taken modulo $3$).  In particular, $\dim(V_1) + \dim(V_2) = 4$.\end{enumerate}
\end{lemma}
The lemma is an avatar of triality for the group $\Spin_8$.  To setup the second lemma, recall that $\Spin_8$ can be defined to be the subgroup of triples $(g_1,g_2,g_3) \in \SO(\Theta)^3$ such that $\tr((g_1x_1) (g_2 x_2) (g_3 x_3)) = \tr(x_1 x_2 x_3)$ for all $x_1, x_2, x_3 \in \Theta$.  As such, $\Spin_8$ sits naturally inside the Levi subgroup $L$ of $Q$, by acting trivially on the coordinate spaces $J(c_i)$ and by $g_i^{-1}$ on the coordinate spaces $J(a_i)$.
\begin{lemma} One has the following facts regarding the $\Spin_8$ orbits on isotropic subspaces of $\Theta$.
\begin{enumerate}
\item Via any of the three projections $\Spin_8 \rightarrow \SO(\Theta)$, the group $\Spin_8$ acts with one orbit on the isotropic lines $\ell \subseteq \Theta$.
\item Via any of the three projections $\Spin_8 \rightarrow \SO(\Theta)$, the group $\Spin_8$ acts with one orbit on the isotropic two-dimensional subspaces of $\Theta$.
\item If $i \in \Z/3$ is an index, $V_i, V_{i+1}$ are subspaces of $\Theta$ with $V_i \cdot V_{i+1} = 0$, and $g = (g_1, g_2, g_3) \in \Spin_8$, then $(g_i V_i) \cdot (g_{i+1} V_{i+1}) = 0$.\end{enumerate} \end{lemma}

\subsection{The $H$-orbits on $\calV$}\label{ssec:Vorbits}
In this subsection, we discuss the $H=\G_2 N$ orbits on the flag variety $\calV$.
Note that the Levi subgroup $L$ (see subsection \ref{subsec:D4parab}) of $Q$ takes each of the types of orbits above to themselves, but moves around the isotropic two-dimensional subspaces $V_i$ in item \eqref{item222} and moves around the isotropic line $\ell$ in items (2) to (7).  Thus to compute the $\G_2 N$ orbits on $\calV$, we need to compute the $\G_2$ orbits on isotropic two-dimensional subspaces $\NullU \subseteq \Theta$ and on isotropic lines $\ell \subseteq \Theta$.

\begin{proposition}\label{prop:G2lines}
The group $\G_2$ acts with two orbits on the set of isotropic lines in $\Theta$.  These orbits are characterized by whether the line $\ell$ is traceless or not, i.e., whether $\ell \subseteq V_7$ or not.  In the Zorn model, the two orbits are represented by the lines $\ell = F\mm{0}{v}{0}{0}$ with $v \neq 0$, and $F\epsilon_1 = \mm{F}{0}{0}{0}$.  The stabilizer of the line $\ell$ is the parabolic subgroup $P(\ell)$ of $\G_2$, while the stabilizer of the line $F \epsilon_1$ is the subgroup $\SL_3$ acting component-wise on the Zorn model.
\end{proposition}

\begin{proof}
It is clear that the two lines $\ell$ and $F\epsilon_1$ are in different $\G_2$ orbits, since $\ell \subseteq V_7$.  That all the isotropic lines in $V_7$ are in one $\G_2$ orbit, and their stabilizer is a parabolic subgroup was stated in Proposition \ref{prop:G2parab}.  Thus we must only check that all the isotropic lines not contained in $V_7$ are in one orbit, and the stabilizer statement for the line $F\epsilon_1$.

We first check that isotropic lines $\ell$ that are not traceless are in one $\G_2$-orbit.  For this, suppose $\ell$ is such a line, and take $y \in \ell$ with $\tr(y) =2$.  Then $y = 1 + y_1$, with $y_1 \in V_7$ and $n(y_1) = -1$.  Since $\G_2$ acts transitively on isomorphic quadratic \'etale subalgebras of $\G_2$, we may move $y_1$ to $\mm{1}{}{}{-1}$.  Thus we can assume $\ell$ is spanned by $\epsilon_1$, as desired.

It is clear that $\SL_3$ is contained in the stabilizer $S(\ell)$.  To see that it is exactly the stabilizer, note that $S(\ell)$ must fix $\epsilon_1$ (since it preserves the trace), and must furthermore fix $\Ann_{R}(\epsilon_i)$, $\Ann_{L}(\epsilon_i)$ for $i = 1,2$, since it commutes with conjugation on $\Theta$.  Taking the intersection of various of these subspaces shows that $S(\ell)$ fixes the components $V_3$ and $V_3^\vee$ of the Zorn model.  Since $S(\ell)$ preserves the bilinear form on $\Theta$, the action of $S(\ell)$ on $V_3^\vee$ is determined by that on $V_3$, and thus $S(\ell) \subseteq \GL_3 = \mathrm{Aut}(V_3)$.  Finally, the trilinear form $\tr(x_1(x_2x_3))$ on $\Theta$ restricted to $V_3$ is the determinant: $(v_1,v_2,v_3) \mapsto v_1 \wedge v_2 \wedge v_3$.  Since $S(\ell)$ must stabilize this, we get $S(\ell)\simeq \SL_3$, as desired.
\end{proof}

We now discuss the $\G_2$-orbits on the isotropic two-dimensional subspaces $V_2 \subseteq \Theta$.
\begin{lemma}\label{G2D4orbits}
There are five orbits of $\G_2$ on isotropic two-dimensional subspaces $V_2$ of $\Theta$, which are characterized as follows:
\begin{enumerate}
\item $V_2$ traceless and null
\item $V_2$ traceless and not null
\item $V_2$ not traceless and is left-null, but not right-null
\item $V_2$ not traceless and is right-null, but not left-null
\item $V_2$ not traceless, and neither right nor left-null.
\end{enumerate}
\end{lemma}

Suppose $v\in V_3$, $\phi \in V_3^\vee$ and $\phi(v) = 0$.  Then examples of such spaces are, in order,
\begin{enumerate}
\item $V_2$ spanned by $\phi$ and $v$;
\item $V_2$ spanned by $e_1$ and $e_2$;
\item $V_2$ spanned by $\mm{1}{0}{0}{0}$ and $v$;
\item $V_2$ spanned by $\mm{1}{0}{0}{0}$ and $\phi$;
\item $V_2$ spanned by $\mm{1}{0}{0}{0}$ and $v+\phi$. \end{enumerate}
\begin{proof}[Proof of Lemma \ref{G2D4orbits}]
It is clear that the examples are spaces of each kind, and thus the five types of spaces $V_2$ exist.  Furthermore, it is clear that the characterizing features of these orbits, e.g., ``traceless and null'', etc are $\G_2$-invariants.  Thus there are at least five orbits of $\G_2$ on the two-dimensional isotropic subspaces of $\Theta$.  That these are the only five orbits can be checked directly.  However, that there are only five orbits follows from \cite[Lemma 2.1]{jiang}, so we omit this aspect.
\end{proof}

The stabilizers in $\G_2$ of these two-dimensional isotropic subspaces $V_2 \subseteq \Theta$ were also computed in \cite[Lemma 2.1]{jiang}, so we simply state the result.  Recall that $e_1, e_2, e_3$ denotes our standard basis of $V_3$, and $e_1^*, e_2^*, e_3^*$ the dual basis of $V_3^\vee$.  For the next proposition, define
\[V_1(e_3^*) = \left(\begin{array}{cc} 0 & F e_1, F e_2 \\ 0 & 0 \end{array}\right), V_2(e_3^*) = \left(\begin{array}{cc} F & 0 \\ F e_3^*& 0 \end{array}\right), V_3(e_3^*) = \left(\begin{array}{cc} 0 & 0 \\ Fe_3^* & F \end{array}\right).\]
\begin{proposition}\label{prop:G2Ustab}
Denote $V_2(closed)$ the space spanned by $F e_3^*, F e_1$ and $V_2(open)$ the subspace spanned by $\epsilon_1$ and $e_1 + e_3^*$.  We have the following stabilizers:
\begin{enumerate}
\item The stabilizer in $\G_2$ of $V_2(closed)$ is the Heisenberg parabolic $P(Fe_1 + Fe_3^*)$.
\item One has $V_j(e_3^*) \cdot V_{j+1}(e_3^*) = 0$ (indices taken modulo $3$), and thus each has the same stabilizer.  Denote by $P(e_3^*)$ the parabolic subgroup of $\G_2$ stabilizing the line spanned by $e_3^*$.  Then $P(e_3^*) = M(e_3^*) U(e_3^*)$, with $M(e_3^*) \simeq \GL_2$ the Levi subgroup fixing the decomposition
\[V_7 = F e_3^* \oplus (F e_1 + Fe_2) \oplus F(\epsilon_1 - \epsilon_2) \oplus (Fe_1^* + F e_2^*) \oplus Fe_3\]
and $U(e_3^*) \supseteq U(e_3^*)' \supseteq U(e_3^*)''$ the three-step unipotent radical.  The stabilizer of the $V_j(e_3^*)$ is $M(e_3^*) U(e_3^*)'$.
\item The stabilizer of $V_2(open)$ is contained in the Heisenberg parabolic $P(Fe_1 + Fe_3^*)$.  Define $U(e_1,e_3^*)$ to be the unipotent radical of $P(Fe_1 + Fe_3^*)$, and set $M(e_1,e_3^*)\simeq \GL_2$ the Levi subgroup fixing the decomposition
\[V_7 = (Fe_1 + Fe_3^*) \oplus (Fe_2 + Fe_2^* + F(\epsilon_1 - \epsilon_2)) \oplus (Fe_1^* + Fe_3).\]
The unipotent radical $U(e_1,e_3^*)$ is spanned by $3$ long roots and two short roots.  The stabilizer of $V_2(open)$ is $\GL_1 U^0(e_1,e_3^*)$ where $U^0(e_1,e_3^*)$ is a certain subgroup of $U(e_1, e_3^*)$ defined by a linear relation on the two short roots in $U(e_1,e_3^*)$ and $\GL_1 \subseteq M(e_1,e_3^*)$ is the subgroup of $\GL_2$ acting as $e_1 \mapsto t e_1$, $e_3^* \mapsto t e_3^*$, $t \in \GL_1$.
\end{enumerate}
\end{proposition}

\begin{proof}
That these are the stabilizers is essentially \cite[Lemma 2.1]{jiang}.
\end{proof}

Putting it all together, we have the following proposition.
\begin{proposition}\label{prop:G2Norbits} The group $\G_2 N$ acts on the flag variety $\calV$ with $17$ orbits.  The $17$ orbits are as follows:
\begin{enumerate}
\item Five $(2,2,2)$ orbits, with representatives given by the isotropic two-space $V_1$ equal to each of the five isotropic two-spaces in Lemma \ref{G2D4orbits}.
\item Six right $(4,1,1)$ orbits, with representatives given by the isotropic line $\ell$ in each of the coordinate spaces $J(a_1), J(a_2)$ and $J(a_3)$, and $\ell$ equal to each of the two isotropic lines in Proposition \ref{prop:G2lines}.
\item Six left $(4,1,1)$ orbits, with representatives given by the isotropic line $\ell$ in each of the coordinate spaces $J(a_1), J(a_2)$ and $J(a_3)$, and $\ell$ equal to each of the two isotropic lines in Proposition \ref{prop:G2lines}.
\end{enumerate}
Denote by $V$ one of the totally singular spaces above, and $H_V$ its stabilizer in $H = \G_2N = H_0 N$.  Then $H_V = (H_V \cap \G_2)(H_V \cap N)$, and the stabilizers $H_V \cap \G_2 = H_V \cap H_0$ are given by Proposition \ref{prop:G2Ustab} in the case of the case of the five $(2,2,2)$ orbits, and given by $P(v)$ and $\SL_3$ in the case of the $(4,1,1)$ orbits.
\end{proposition}
\begin{proof} The only thing that we have not yet proved is that $H_V = (H_V \cap \G_2)(H_V \cap N)$ for the above spaces $V$.  But this is clear by considering the fact that the $V$ above are a direct sum of their intersections with the coordinate spaces $J(c_i)$, $J(a_i)$, $i = 1,2,3$.\end{proof}

\subsection{$N$-stabilizers of totally singular spaces}
%In order to compute the period $P_{G_2}(\Lambda^{T}E(g,\phi,s))$, we will also need to know some information about the stabilizers $H_V \cap N$ for the spaces $V$ of Proposition \ref{prop:G2Norbits}.  We now compute some of these stabilizers.

We will require the following lemma.
\begin{lemma}\label{lem:perp*} Suppose $V_1, V_2, V_3$ are a triple of $(2,2,2)$ spaces with $V_j \cdot V_{j+1} = 0$.  Then $V_1 \cdot V_2^\perp \subseteq V_3^*$ and $V_2^\perp \cdot V_3 \subseteq V_1^*$ (and cyclic permutations of these).\end{lemma}
\begin{proof} This is surely well-known, so we don't give a detailed proof.  However, the reader wishing to check this themselves can simply check it for each of the five cases of Lemma \ref{G2D4orbits}, since the statement of the lemma is $\G_2$-invariant.\end{proof}

If $(x,y;z) \in N(\Theta)$, recall that we denote $n(x,y,z)$ the corresponding element in $\E_6$ as defined above.  Also recall that if $V$ is one of the totally singular six-dimensional spaces in $J$, $V(a_3) = V \cap J(a_3)$, $V(a_1) = V \cap J(a_1)$, etc.
\begin{lemma}\label{lem:222Stab} The conditions for $n(x,y,z)$ to stabilize one of the $(2,2,2)$ orbits are the following:
\begin{itemize}
\item $x \in V(a_3)^\perp$,
\item $y \in V(a_1)^\perp$,
\item $z^* \in V(a_2)^\perp$.
\end{itemize}
If $x \in V(a_3), y \in V(a_1)$ and $z^* \in V(a_2)$, then $n(x,y,z)$ acts trivially on the $(2,2,2)$ space.
\end{lemma}

\begin{proof}
One checks immediately that if $x \in V(a_3),y\in V(a_1)^{\perp}$ and $z^* \in V(a_2)^\perp$, then $n(x,y,z)$ acts as $1$ on the $(2,2,2)$-space.  Also, one sees quickly that for $n(x,y,z)$ to stabilize a $(2,2,2)$-space, it is necessary that $x \in V(a_3)^\perp, y \in V(a_1)$ and $z^* \in V(a_2)$.  But now by Lemma \ref{lem:perp*}, these ``perp'' conditions are also sufficient.  This completes the proof.
\end{proof}

We now consider the $N$ stabilizers for the $(4,1,1)$-orbits.
\begin{proposition}\label{prop:411Stab} Suppose $V$ is one of the $(4,1,1)$-orbits.
\begin{enumerate}
\item Suppose $V$ is one of the orbits with $c_1 = c_2 = a_3 = 0$.  Then $n(x,y,z)$ stabilizes $V$ if and only if $V(a_2) \cdot x \subseteq V(a_1)^*$, and acts trivially on $V$ if and only if $V(a_2) \cdot x = 0$, $(V(a_1),y) = 0$, and $(V(a_2), z^*) = 0$.
\item Suppose $V$ is one of the orbits with $c_1 = c_3 = a_2 = 0$.  Then $n(x,y,z)$ stabilizers $V$ if and only if $y \in V(a_1)$ and $V(a_3)^* z \subseteq V(a_1)$.  The element $n(x,y,z)$ acts trivially on $V$ if and only if $x \in V(a_3)^\perp$, $y = 0$, and $z^* V(a_3) = 0$.
\item Suppose $V$ is one of the orbits with $c_2 = c_3 = a_1 = 0$.  Then $n(x,y,z)$ stabilizes $V$ if and only if $x \in V(a_3)$, $z^* \in V(a_2)$ and $V(a_3) \cdot y \subseteq V(a_2)^*$.  Such a $n(x,y,z)$ acts as the identity on $V$ if and only if $z= x= 0$ and $V(a_3) \cdot y = 0$.
\end{enumerate}
\end{proposition}
\begin{proof} The proof of the first item is immediate from the formulas for the $N(\Theta)$ action.  For the second item, to see that $y \in V(a_1)$, look at how the $a_1$-coordinate in $V$ changes.  Since $y \in V(a_1)$, it follows that $V(a_3) y = 0$ automatically, and the rest of the stabilizer claim follows immediately.  The conditions for $n(x,y,z)$ to act as the identity are similarly immediate.

Finally, consider the third item, and the claim about the stabilizer.  Looking at the $a_2$-coordinate, one sees right away that $z^* \in V(a_2)$ and $V(a_3) y \subseteq V(a_2)^*$.  Looking at the $a_1$-coordinate, one sees $V(a_2) x = 0$, which happens if and only if $x \in V(a_3)$.  Now with these conditions, one checks easily that $n(x,y,z)$ does indeed stabilize $V$.

For $n(x,y,z)$ to act as the identity on $V$, one sees that $z=0$ by looking at the $a_2$-coordinate, and that $x=0$ by looking at the $a_3$-coordinate.  By looking again at the $a_2$-coordinate, one concludes that $V(a_3)\cdot y = 0$.  When these conditions are satisfied, it follows immediately that $n(x,y,z)$ acts trivially on $V$. This concludes the proof.
\end{proof}

\section{Eisenstein series and related objects}\label{sec:Eis}
\subsection{General notations}
For the rest part of this paper, $F$ is a number field, $\bar{F}$ is the algebraic closure of $F$, $\BA=\BA_F$ be the ring of adeles.

Let $G$ be a connected reductive algebraic group over $F$, $X(G)$ be the group of rational characters of $G$. We fix a maximal $F$-split torus $A_{0}$ of $G$. Let $P_{0}$ be a minimal parabolic subgroup of $G$ defined over $F$ containing $A_0$, $M_{0}$ be the Levi part of $P_{0}$ containing $A_{0}$ and $U_{0}$ be the unipotent radical of $P_{0}$. Let $\calF(P_{0})$ be the set of parabolic subgroups of $G$ containing $P_{0}$. Elements in $\calF(P_0)$ are called standard parabolic subgroups of $G$. We also use $\calF(A_0)$ to denote the set of parabolic subgroups of $G$ containing $A_0$ (these are the semi-standard parabolic subgroups).

For $P \in \calF(P_{0})$, we have the Levi decomposition $P = MU$ with $U$ be the unipotent radical of $P$ and $M$ be the Levi subgroup containing $A_{0}$. We use $A_{P} \sbs A_{0}$ to denote the maximal $F$-split torus of the center of $M$. Put
\[
 \ago_{0}^{*} = X(A_{0}) \otimes_{\Z}\R = X(M_{0}) \otimes_{\Z}\R
\]
and let $\ago_{0}$ be its dual vector space. Let $\Delta_{0} \sbs \ago_{0}^{*}$ be the set of simple roots
of $A_{0}$ acting on $U_{0}$. The subsets of $\Delta_{0}$ are in a natural bijection with $\calF(P_{0})$.
For $P\in \calF(P_0)$, we use $\Delta_{0}^{P}$ to denote the subset of $\Delta_0$ corresponding to $P$. In particular, we have $\Delta_{0}^{G} = \Delta_{0}$. Set $\ago_{P}$ to be the kernel of $\Delta_{0}^{P}$. If $P = P_{0}$, we write $\ago_{0} = \ago_{P_{0}}$ and similarly in other contexts.

The inclusions $A_{P} \sbs A_{0}$ and $M_{0} \sbs M$ identify $\ago_{P}$
as a direct factor of $\ago_{0}$, we use $\ago_{0}^{P}$ to denote its complement.
Similarly, $\ago_{P}^{*} = X(A_{P}) \otimes_{\Z}\R$ is a direct factor of $\ago_{0}^{*}$
and we use $\ago_{0}^{P,*}$ to denote its complement. The space $\ago_{0}^{P,*}$
is generated by $\Delta_{0}^{P}$.

Let $\Delta_{0}^{\vee} \sbs \ago_{0}^{G}$ be the set of simple coroots given by the theory
of root systems. For $\al \in \Delta_{0}$ we denote $\al^{\vee} \in \Delta_{0}^{\vee}$
the corresponding coroot. We define $\hDelta_{0} \sbs \ago_{0}^{G, *}$ to be the dual basis of
$\Delta_{0}^{\vee}$, i.e. the set of weights. In particular, we get a natural bijection
between $\Delta_{0}$ and $\hDelta_{0}$ which we denote by $\al \mapsto \varpi_{\al}$.
Let $\hDelta_{P} \sbs \hDelta_{0}$  be the set corresponding to $\Delta_{0} \smin \Delta_{0}^{P}$.

For any subgroup $H \sbs G$ let $H(\A)^{1}$ denote the common kernel of all continuous characters of $H(\A)$ into $\R_{+}^{*}$. Fix $K$ a maximal compact subgroup of $G(\A)$ adapted to $M_{0}$. We define the Harish-Chandra homomorphism $H_{P} : G(\A) \to \ago_{P}$ via the relation
\[
 \langle \chi, H_{P}(x) \rangle = |\chi(p)|_{\A}, \quad \forall \chi \in X(P) = \mathrm{Hom}(P,\mathbb{G}_m)
\]
where $x = pk$ is the Iwasawa decomposition $G(\A) = P(\A)K$ and $|\cdot |_{\A}$ is the absolute vaule on the ideles of $\A$. Let $A_{P}^{\infty}$ be the connected component of $\mathrm{Res}_{F/\Q}A_{P}(\R)$. Then $M(\A)^{1}$ is the kernel of $H_{P}$ restricted to $M(\A)$ and we have the direct product decomposition of commuting subgroups $M(\A) = A_{P}^{\infty} M(\A)^{1}$.

For any group $H$ we use $[H]$ to denote $H(F)\back H(\BA)$. Moreover, if $H$ is reductive, we use (assuming the compact subgroup of $H(\A)$ is clear from the context) $[H]^1$ to denote $H(F) \back H(\A)^{1}$.

\subsection{Haar measures}
We fix compatible Haar measure on $G(\BA)$, $G(\BA)^1$ and $A_{G}^{\infty}$. For all unipotent subgroups $U$ of $G$, we fix a Haar measure on $U(\A)$ so that $[U]$ is of volume one. On $K$ we also fix a Haar measure of volume $1$. For any $P\in \calF(A_0)$, let $\Delta_{P}^{\vee}$ be the set projections of $\{\al^{\vee}\}_{\al \in \Delta_{0} \smin \Delta_{0}^{P}}$ onto $\ago_{P}$. Then $\Delta_{P}^{\vee}$ is a basis of $\ago_{P}/\ago_{G}$. We use this basis to define a Haar measure on $\ago_{P}/\ago_{G}$. This choice induces a unique Haar measure on $A_{P}^{\infty}/A_{G}^{\infty}$ such that
\begin{equation}\label{eq:realInt}
 \int_{A_{P}^{\infty}/A_{G}^{\infty}} f(H_{P}(a)) \, da =  \int_{\ago_{P}/\ago_{G}} f(H) \, dH,
 \quad f \in C_{c}^{\infty}(\ago_{P}/\ago_{G}).
\end{equation}
Together with the measure on $A_{G}^{\infty}$, we get a Haar measure on $A_{P}^{\infty}$.

Let $\rho_{P} \in \ago_{P}^{*}$ be the half sum of the weights of the action of $A_{P}$ on $N_{P}$. The above choices induce a unique Haar measure on $M_P(\BA)^1$ such that
\[
 \int_{P(F) \bsl H(\A)}f(h) \, dh =
 \int_{K}\int_{[M]^{1}}   \int_{A_{P}^{\infty}} \int_{[U]}
 e^{\langle -2\rho_{P}, H_{P}(a) \rangle}
 f(uamk) \, du da dm dk
\]
for $f \in C_{c}^{\infty}(P(F) \bsl G(\A))$. We fix this measure on
$M(\A)^{1}$ as well.

\subsection{The computation of $\rho_P$ when $P$ is maximal}\label{ssec:rho}
Let $P \in \calF(P_{0})$ be a maximal parabolic subgroup corresponds to the simple root $\alpha$, i.e. $\{\al\} = \Delta_{0} \smin \Delta_{0}^{P}$.
Let $\varpi$ be the corresponding weight. We have $\rho_{P} \in \ago_{P}^{G,*}$. Since $P$ is maximal, $\ago_{P}^{G,*}$ is one dimensional. Hence there exists a constant $c \in \R$ such that $\rho_{P} = c \varpi$. We want to calculate this constant.

Write $\al^{\vee} = \underline{\al}^{\vee} + \sum_{\gamma \in \Delta_{0}^{P}} a_{\gamma^{\vee}} \gamma^{\vee}$
with respect to the direct sum decomposition $\ago_{0}^{G} = \ago_{P}^{G} \oplus \ago_{0}^{P}$.
It is known that $\rho_{0} = \rho_{P_0}=\sum_{\varpi \in \hDelta_{0}} \varpi$.
Hence
\[
 c = \langle \rho_{P}, \underline{\al}^{\vee} \rangle =
\langle \rho_0, \underline{\al}^{\vee} \rangle = 1 - \sum_{\gamma \in \Delta_{0}^{P}} a_{\gamma^{\vee}}.
\]
We are reduced to calculating the constants $a_{\gamma^{\vee}}$. Let $n = |\Delta_{0}|$ be the rank of $G$.
Let $C$ be the Cartan matrix of $G$, it is an $n \times n$ matrix with entries
$c_{\al, \beta^{\vee}} = \langle \al, \beta^{\vee} \rangle$ where $\al \in \Delta_{0}$
and $\beta^{\vee} \in \Delta_{0}^{\vee}$. Let $C_{\al}$ be the $(n-1) \times (n-1)$
matrix obtained from $C$ by removing the $\al$-row and $\al^{\vee}$-column
(the Cartan matrix of the root system of $M$).
If we denote $v_{\al} \in \R^{n-1}$ the column vector $(a_{\gamma^{\vee}})_{\gamma \in \Delta_{0}^{P}}$
and $w_{\al} \in \R^{n-1}$  the column vector
$(\langle \gamma, \al^{\vee} \rangle)_{\gamma \in \Delta_{0}^{P}}$, we clearly have
\[
 v_{\al} = C_{\al}^{-1}w_{\al}.
\]

We specialize to two cases that will be needed in later sections. \textbf{Let $G$ be the split, simply connected reductive group of type $\E_6$.} Let $\Delta_{0} = \{\al_{i}\}_{i \in \{1, \ldots, 6\}}$ be the set of simple roots whose Cartan matrix $C_{ij} = \langle \al_{i}, \al_{j}^{\vee} \rangle$ is of the form
\[
 C = \begin{pmatrix}
      2 & -1 & 0 & 0 & 0 & 0 \\
      -1 & 2 & -1 & 0 & 0 & 0 \\
      0 & -1 & 2 & -1 & 0 & -1 \\
      0 & 0 & -1 & 2 & -1 & 0 \\
      0 & 0 & 0 & -1 & 2 & 0 \\
      0 & 0 & -1 & 0 & 0 & 2
     \end{pmatrix}.
\]
Let $P$ be the standard parabolic subgroup corresponding to $\Delta_{0} \smin \{\al_{6}\}$.
Its Levi subgroup is then of type $A_{5}$.
We have that
\[
 C_{\al_{6}}^{-1} =
    \begin{pmatrix}
      5/6 & 2/3 & 1/2 & 1/3 & 1/6 \\
      2/3 & 4/3 & 1 & 2/3 & 1/3 \\
      1/2 & 1 & 3/2 & 1 & 1/2  \\
      1/3 & 2/3 & 1 & 4/3 & 2/3  \\
      1/6 & 1/3 & 1/2 & 2/3 & 5/6
     \end{pmatrix}, \quad
     w_{\al_{6}} =
         \begin{pmatrix}
      0\\
      0 \\
      -1 \\
      0 \\
      0
     \end{pmatrix}.
\]
Consequently, we get
\[
 \rho_{P} = \frac{11}{2} \varpi_{6}
\]
where $\varpi_{6} \in \hDelta_{0}$ corresponds to $\al_{6}$.

\begin{remark}
We know that $M \cong \{(\la, g) \in GL_{1} \times GL_{6}  \ | \ \la^{3} = \det g\}$.
Let $\la$ denote the character of $M$ corresponding to the projection on the first factor.
Let's fix this isomorphism so that $\la$ restricted to  $A_{P}$ acts via a positive powers on $U$.
Since $\E_6$ is simply connected, $\varpi_{6}$ is an indivisible character of $A_{0}$
and we conclude that $\la = \varpi_{6}$.
\end{remark}

\textbf{Let $G$ be the split reductive group of type $\G_2$.} Let $\Delta_{0} = \{\al_{1}, \al_{2}\}$ be the set of simple roots whose Cartan matrix $C_{ij} = \langle \al_{i}, \al_{j}^{\vee} \rangle$ is of the form
\[
 C = \begin{pmatrix}
      2 & -1 \\
      -3 & 2
     \end{pmatrix}.
\]
In particular, $\al_{2}$ is the long root. Let $P$ be the standard parabolic subgroup corresponding to $\Delta_{0} \smin \{\al_{2}\}$. Its Levi subgroup is of type $A_{1}$ and the unipotent radical is a two-step unipotent subgroup. Obviously $C_{\al_{2}}^{-1} = \frac{1}{2}$ and $w_{\al_{2}} = -1$ which imply that
\[
 \rho_{P} = \frac{3}{2} \varpi_{2}
\]
where $\varpi_{2} \in \hDelta_{0}$ corresponds to $\al_{2}$.

\begin{remark}
Note that $M \cong GL_{2}$
and we can fix this isomorphism so that the determinant character $\det$ acts via positive powers on $U$.
Through this identification we then have $\varpi_{2} = \det$.
\end{remark}

\subsection{Eisenstein series}\label{ssec:eis}
Let $P=MU$ be a parabolic subgroup of $G$. Given a cuspidal automorphic representation $\pi$ of $M(\A)$, let $\calA_{\pi}$ be the space of automorphic forms $\phi$ on $N(\A)M(F) \bsl G(\A)$ such that $M(\A)^1 \ni m \mapsto \phi(mg) \in L^{2}_{\pi}([M]^{1})$ for any $g \in G(\A)$, where $L^{2}_{\pi}([M]^{1})$ is the $\pi$-isotypic part of $L^{2}([M]^{1})$, and such that
\[
 \phi(ag) = e^{\langle \rho_{P}, H_{P}(a) \rangle} \phi(g), \quad \forall g \in G(\A),\;a\in A_{P}^{\infty}.
\]
Suppose that $P$ is a maximal parabolic subgroup. Let $\varpi \in \hDelta_{P}$ be the corresponding weight. We then define
\[
 E(g, \phi, s) = \sum_{\delta \in P(F) \bsl G(F)}
 \phi(\delta g)e^{\langle s\varpi, H_{P}(\delta g )\rangle}, \quad s \in \C, \
 g \in G(\A).
\]
The series converges absolutely for $s \gg 0$ and admits a meromorphic continuation to all $s\in \BC$.

Suppose moreover that $M$ is stable for the conjugation by the simple reflection in the Weyl group of $G$ corresponding to $P$. We have in this case the intertwining operator $M(s) : \calA_{\pi} \to \calA_{\pi}$ that satisfies $E(M(s)\phi, -s) = E(\phi, s)$ and
\[
E(g, \phi, s)_{P} = \phi(g)e^{\langle s \varpi, H_{P}(g) \rangle} +
 e^{\langle -s\varpi, H_{P}(g) \rangle }M(s)\phi(g) ,\quad g \in G(\A)
\]
where $E( \cdot , \phi, s)_{P}$ is the constant term of $E(\cdot , \phi, s)$ along $P$
\[
 E(g, \phi, s)_{P} :=  \int_{[U]}E(ug, \phi, s) \, du.
\]

When the Eisenstein series $E(g, \phi, s)$ has a pole at $s=s_0$, the intertwining operator also has a pole at $s=s_0$, we use $Res_{s=s_0}E(g, \phi, s)$ (resp. $Res_{s=s_0}M(s)$) to denote the residue of the Eisenstein series (resp. intertwining operator).
Recall that the Eisenstein series, their derivatives and residues
are of moderate growth. Moreover, for $s$ in the domain of holomorphy of $E(\phi, s)$
we have for all $X$ in the universal enveloping algebra of the complexification of the Lie algebra of $G$ a bound
\[
X \ast E(g, \phi, s) \le c(s) (\inf_{\gamma \in G(F)}\| \gamma x\|_{G})^{N}, \quad g \in G(\A)
\]
for some $N > 0$ and some locally bounded function $c$ on $\C$ where $\|x\|_{G}$ is a norm on $G(\A)$. This last fact allows for bounds uniform
in $s$ as long as it is confined to compact sets.

\subsection{Truncation operator}
We continue assuming that $P$ is maximal.
We identify the space $\ago_{P}^{G}$ with $\R$ so that
$T \in \R$ corresponds to an element whose pairing with $\varpi \in \ago_{P}^{*}$
is $T$. We will assume this isomorphism is measure preserving. Let $\htau_{P}$ be the characteristic function of
\[
 \{H \in \ago_{P} \ | \ \varpi(H) > 0 \ \forall \varpi \in \hDelta_{P}\}.
\]
Given a locally integrable function $F$ on $G(F) \bsl G(\A)$
we define its truncation as follows
\[
 \La^{T}F(g) = F(g) - \sum_{\delta \in P(F) \bsl G(F)}
 \htau_{P}(H_{P}(\delta g) - T) \int_{[U]}F(u\delta g) \, du, \quad g \in G(F) \bsl G(\A),
\]
where $T \in \R$ and the sum is actually finite.

\section{Computation of $\mathcal{P}_{\G_2} (\La^T E(\phi,s))$ and proof of Theorem \ref{main}}\label{sec:PG2}
In this section we compute formally the period $\mathcal{P}_{\G_2} (\La^T E(\phi,s))$ defined in section \ref{ssec:periodDef},
where $E(\phi,s)$ is an Eisenstein series
on $\E_6$ associated to the maximal parabolic subgroup
of type $A_5$ and a cuspidal representation of its Levi subgroup.
The computation is performed in sections \ref{ssec:eisSerExpl} through \ref{ssec:GRperiod} culminating in
Proposition \ref{prop:truncIdentity}.
This proposition is then used to prove Theorem \ref{main} in section \ref{ssec:proofMain}.
The results depend on Proposition \ref{prop:absConv} which will be proven in next section.

\subsection{The Eisenstein series}\label{ssec:eisSerExpl}
We first single out the parabolic $P$ that we use to define the Eisenstein series $E(\phi,s)$.  To do so, consider $\Theta$ as formed by the Cayley-Dickson construction out of $D = M_2(F)$ and $\gamma = 1$, in the notation of section \ref{subsubsec:CD}.  Then, define $\NullU \subseteq \Theta$ as the two-dimensional subspace consisting of elements $(x,y) = (0,\mm{*}{*}{0}{0})$.  It is clear that $\NullU \subseteq V_7$, and one checks immediately that $\NullU$ is null.  Thus, the six-dimensional space $V(\NullU) \subseteq J$, see \eqref{eqn:V(U)}, is totally singular.  We define $P = P(V(\NullU))$ to be its stabilizer inside $\E_6$.  From Proposition \ref{prop:A5parab}, $P$ has reductive quotient of type $A_5$.

We now explicitly describe a Levi subgroup $M$ of $P$. Set $D = M_2(F)$, so that $\Theta = D \oplus D$ as in the Cayley-Dickson construction.  To describe the Levi subgroup $M$ of $P$, it is convenient to write $J = H_3(\Theta)$ as a direct sum of two pieces, which corresponds to the direct sum $\Theta = D \oplus D$.  Namely, $J = H_3(D) \oplus D^3$, with this $D^3 = M_{2,6}(F)$ considered as $1 \times 3$ row vectors in $D$, or $2 \times 6$ matrices over $F$.  In this decomposition, if $(X,v) \in H_3(D) \oplus D^3$, then one finds $n((X,v)) = n(X) + vXv^*$.  Here the notation is as follows.  The $n$ on the left-hand side is the norm on $J$, the $n$ on the right-hand side is the norm cubic norm on $H_3(D)$ (given by the same formula as in \eqref{eqn:nX}, and $v^*$ is the column vector in $D^3$ given by applying transpose-conjugate to $v$.  Thus $vXv^*$ is an element of $D$ fixed by the conjugation $*$, so it is in $F$.

In this decomposition, the six-dimensional subspace $V(\NullU)$ becomes the set of $(X,v) = (0, v)$, $v \in M_{2,6}(F)$, where the bottom row of $v$ is $0$.  Now, recall from Proposition \ref{prop:A5parab} that $M \simeq \{ (\lambda, g) \in \GL_1 \times \GL_6: \lambda^3 = \det(g)\}$.  We let $M$ act on $J$ via
\begin{equation}\label{eqn:MactJ}(X, v) \mapsto \left(\lambda^{-1} g^* X g, \mm{1}{}{}{\lambda} v (\,^*g^{-1})\right).\end{equation}
Here $\,^*g$ is the transpose conjugate of $g$ considered as an element of $M_3(D) = M_6(F)$.  Note that this action of $M$ on $J$ preserves $V(\NullU)$, and one checks immediately that this action preserves the norm.  Thus $M \subseteq P \subseteq \E_6$, and we use this $M$ as a Levi subgroup of $P$. We have a natural homomorphism from $M$ to $\GL_6$ given by $(\lambda,g)\mapsto g$. Under this normalization, the modular character is given by $\delta_{P}((\lambda,g)) = |\lambda|^{11}$.

\subsection{First step}
Given a cuspidal automorphic representations $\pi$ of $\GL_6(\BA)$, we can view it as a cuspidal automorphic representation of $M(\BA)$ via the natural homomorphism $M\rightarrow \GL_6$ described above. For $\phi\in \calA_{\pi}$ and $s\in \BC$ let $E(\phi, s)$ be the corresponding
Eisenstein series on $\E_6$.

\begin{lemma}\label{lem:cvgLaTE}
For all $\phi \in \calA_{\pi}$ and $s\in \BC$
the integral defining the period $\mathcal{P}_{\G_2}(\La^{T}E(\phi,s))$ converges absolutely
and uniformly for $s$ in a compact subset of the domain of holomorphy of $E(\phi,s)$.
\end{lemma}

\begin{proof}
Note that the function $x \mapsto \La^{T}E(x,\phi,s)$
equals the truncated Eisenstein series via the operator introduced in \cite{arthur}.
This is because $E(\phi,s)$ is a cuspidal Eisenstein series induced from a maximal, self-dual Levi subgroup
of $\E_6$.
It follows from Lemma 1.4 of \emph{loc. cit.}
that
the function $x \mapsto \La^{T}E(x,\phi,s)$
is rapidly decreasing on $[G]=[\E_6]$.
The lemma is thus a direct consequence of Proposition A.1.1(ix) of \cite{B17} together with the fact that $\G_2N\back \E_6$ is quasi-affine.
\end{proof}

Let $\calV$ be the flag variety as in the section \ref{sec:orbits}.
Identify $\calV/H$ with a set of representatives so that the subspace defining the parabolic subgroup $P$ is one of them.
Let $V \in \calV/H$.
Fix $\gamma_{V} \in \E_6(F)$ such that $V = V(\NullU) \gamma_{V}$.
Let $P(V)$ be the parabolic subgroup of $\E_6$ stabilizing $V$, $U(V)$ the unipotent radical of $P(V)$, and $M(V)$ the reductive quotient of $P(V)$.
By definition, the group $H_V$, the stabilizer of $V$ in $H$, is contained in $P(V)$.
We set $U_H(V) = U(V) \cap H_V$, and $M_H(V)$ the image of $H_V$ in $M(V)$.
We sometimes abuse notation and also write $M(V)$ for the Levi subgroup of $P(V)$ equal to $\gamma_{V}^{-1} M \gamma_{V}$,
where $M$ is the Levi subgroup of $P$ specified above.

Let $V \in \calV/H$.
Set
\[
I_{V}(\phi,s) := \int_{H_{V}(F) \bsl H(\A)}(1-\htau_{P}(H_{P}(\gamma_{V}h) - T))\phi(\gamma_{V}h)e^{\langle s\varpi, H_{P}(\gamma_{V}h) \rangle}\xi(h)\,dh
\]
and
\[
J_{V}(\phi,s) :=\int_{H_{V}(F) \bsl H(\A)}\htau_{P}(H_{P}(\gamma_{V}h) - T)M(s)\phi(\gamma_{V}h)e^{\langle -s\varpi, H_{P}(\gamma_{V}h) \rangle}\xi(h)\,dh.
\]
The proof of the following proposition will be given in the next section.
\begin{proposition}\label{prop:absConv}
For all $s \in \C$ such that $\Rel(s)$ is sufficiently large
and all $T \in \R$ sufficiently large, the integrals defining $I_{V}(\phi,s)$
and $J_{V}(\phi,s)$
converge absolutely.
\end{proposition}

Unfolding the Eisenstein series and taking the above proposition for granted we get
\begin{equation}\label{eq:periodDecomp}
\calP_{\G_2}(\La^{T}E(\phi,s)) = \sum_{V \in \calV/H} I_{V}(\phi,s) + J_{V}(\phi,s)
\end{equation}
which justifies the introduction of the integrals $I_{V}$ and $J_{V}$. In the following subsections
\ref{ssec:vanish1} and \ref{ssec:vanish2} we show that $I_{V}(\phi,s) = J_{V}(\phi,s) = 0$
unless $V = V(\NullU)$.

\subsection{Vanishing of most orbits}\label{ssec:vanish1} In this subsection, we show that $14$ of the $17$ integrals $I_V(\phi,s)$ and $J_V(\phi,s)$ vanish.  More specifically suppose that the character $\xi$ is nontrivial on $U_H(V) \cap N$.  It follows that the integrals $I_V(\phi,s)$ and $J_V(\phi,s)$ vanish.  In this subsection we check that $14$ of the $17$ integrals $I_V(\phi,s)$ and $J_V(\phi,s)$ vanish for this reason.

\begin{lemma}
Suppose that $V$ is one of the $(2,2,2)$ orbits (cf. \ref{prop:G2Norbits}), which is not the closed orbit, i.e., $V \neq V(\NullU)$.
Then the integrals $I_V(\phi,s)$ and $J_V(\phi,s)$ vanish.
\end{lemma}

\begin{proof}
An element of $N$ is in $U_H(V) \cap N$ if and only if it acts trivially on $V$.  By Lemma \ref{lem:222Stab}, it follows that there is always an element $n(x,0;0)$ or $n(0,y;0)$ acting trivially on $V$ with either $\tr(x) \neq 0$ or $\tr(y) \neq 0$.  Thus the character $\xi$ is nontrivial on $U_H(V) \cap N$ and the lemma follows.
\end{proof}

\begin{lemma}
Suppose $V$ represents one of the $(4,1,1)$ (cf. \ref{prop:G2Norbits}) orbits with either $V(c_3) \neq 0$ or $V(c_2) \neq 0$, i.e., $V$ is of the form \eqref{item411R1}, \eqref{item411R3}, \eqref{item411L1} or \eqref{item411L2}.  Then the integrals $I_V(\phi,s)$ and $J_V(\phi,s)$ vanish.
\end{lemma}

\begin{proof}
Suppose first $V$ is such that $V(c_3) \neq 0$.  Then by Proposition \ref{prop:411Stab}, $n(0,y;0)$ acts trivially on $V$ so long as $y \in V(a_1)^\perp$.  But $V(a_1)^\perp$ contains a $4$-dimensional isotropic subspace, so the trace is not identically $0$ on it.  Hence $I_V(\phi,s)$ and $J_V(\phi,s)$ vanish.  Similarly, if $V$ is one of the orbits with $V(c_2) \neq 0$, then $n(x,0;0)$ acts trivially on $V$ so long as $x \in V(a_3)^\perp$, and one again gets vanishing for the same reason.
\end{proof}

\begin{lemma}
Suppose $V$ is one of the $(4,1,1)$ orbits with $V(c_1) \neq 0$.  If $V(a_3) = \ell$ is an isotropic line, then the integrals $I_V(\phi,s)$ and $J_V(\phi,s)$ vanish.
\end{lemma}

\begin{proof}
From Proposition \ref{prop:411Stab}, the element $n(x,y;z)$ acts as $1$ on such a $V$ if and only if $x=z=0$ and $V(a_3) \cdot y = 0$.  If $V(a_3) = \ell$, then the set of such $y$ is $\Ann_{R}(\ell)$, which is a four-dimensional isotropic subspace of $\Theta$.  Thus the trace is nonzero on it, which implies that $I_V(\phi,s)$ and $J_V(\phi,s)$ vanish.
\end{proof}

\subsection{Formal vanishing of another two orbits}\label{ssec:vanish2}
In the previous subsection we showed that $14$ of the $17$ integrals $I_V(\phi,s)$ and $J_V(\phi,s)$ vanish.  In this subsection, we show that the integrals $I_V(\phi,s)$ and $J_V(\phi,s)$ also vanish for another $2$ orbits. More specifically, the $V$'s we consider are the two $(4,1,1)$ orbits with $V(c_1) \neq 0$ and $V(a_2)$ an isotropic line. We begin with the following lemma.

\begin{lemma}\label{lem:UVtriv}
Suppose $V$ is one of the $(4,1,1)$ orbits with $V(c_1)\neq 0$ and $V(a_2) = \ell$ an isotropic line.  Then the subgroup of $N$ acting as the identity on $V$ is trivial.
\end{lemma}

\begin{proof}
From Proposition \ref{prop:411Stab}, the element $n(x,y;z)$ acts as $1$ on such a $V$ if and only if $x=z=0$ and $V(a_3) \cdot y = 0$.   If $V(a_2) = \ell$, then $V(a_3) = \Ann_{R}(\ell)$.  But $\Ann_{R}(\ell)$ only has a \emph{left} annihilator (namely, $\ell$); its right annihilator is $0$.  Thus only the identity in $N$ acts trivially on $V$ in these cases.
\end{proof}

For the rest of this subsection, $V$ denotes a representative of one of the $(4,1,1)$ oribts with $V(c_1) \neq 0$ and $V(a_2) = \ell$ an isotropic line.  As we mentioned before, we extend the character $\xi$ on $N$ to $N \G_2$ by making it trivial on $\G_2$.  By Lemma \ref{lem:UVtriv}, the integrals $I_V(\phi,s)$ and $J_V(\phi,s)$ have an inner integral over $\left(M_H(V)(F)\backslash M_H(V)(\A)\right)^{1}$.  Hence in order to show the integrals $I_V(\phi,s)$ and $J_V(\phi,s)$ vanish, it suffices to show that the integral
\[K_V(\varphi) := \int_{\left(M_H(V)(F)\backslash M_H(V)(\A)\right)^{1}}{\xi(x)\varphi(x)\,dx}\]
vanishes for any cusp form $\varphi$ on $M(V)$.  The purpose of the rest of this subsection is to prove the vanishing of $K_V(\varphi)$. The following lemma computes the image of $N$ in $M_H(V)$.

\begin{lemma}\label{lem:4112S}
Suppose $V$ is as above, a $(4,1,1)$ orbit with $V(c_1) \neq 0$, and $V(a_2) = \ell$ an isotropic line.  Then $n(x,y;z)$ stabilizes $V$ if and only if $x \in \Ann_{R}(\ell)$, $z^* \in \ell$, and $y \in \Ann_{L}(\ell)$.
\end{lemma}

\begin{proof}
The statements about $x$ and $z^*$ were proved in Proposition \ref{prop:411Stab}.  The condition on $y$ was determined to be $V(a_3) \cdot y \subseteq V(a_2)^*$, or, in other words, $\Ann_{R}(\ell) \cdot y \subseteq \ell^*$.  However, the subset of $y \in \Theta$ satisfying this condition is precisely $\Ann_{L}(\ell)$.  This completes the proof of the lemma.
\end{proof}

Suppose $x,y,z$ are as specified in Lemma \ref{lem:4112S}, i.e., $x \in \Ann_{R}(\ell)$, $z \in \ell^{*}$, and $y \in \Ann_{L}(\ell)$.  Then $n(x,y;z)$ acts on
\[V = V(a_1) \oplus V(a_3) \oplus V(a_2)= F \oplus \Ann_{R}(\ell) \oplus \ell\]
as matrices of the form
\[\left(\begin{array}{c|c|c} 1 & * & * \\ \hline 0 & 1_4 & * \\ \hline 0&0&1 \end{array}\right).\]
More precisely, recall that $N$ acts on the six-dimensional space $V$ on the right.  We order a basis of $V$ as $V = F \oplus \Ann_{R}(\ell) \oplus \ell$, for some ordered basis of $\Ann_{R}(\ell)$.  For $y \in \Ann_{L}(\ell)$, denote by $\phi_y \in \mathrm{Hom}(\Ann_{R}(\ell),\ell)$ the map $\phi_y(a_3) = (a_3 y)^*$.  Then, $n(x,y;z)$ acts on $V$ via the matrix
\[\left(\begin{array}{c|c|c} 1 & x & z^* \\ \hline 0 & 1_4 & \phi_y \\ \hline 0&0&1 \end{array}\right).\]
We now check that the integral $K_V(\varphi)$ vanishes for each of the two cases, $\ell$ being traceless or not.

\begin{lemma}
Suppose that the isotropic line $\ell$ is contained in $V_7$.  Then $K_V(\varphi)$ vanishes.
\end{lemma}

\begin{proof}
For concreteness, assume $\ell = e_3^* \in \Theta$.  Then $\Ann_{R}(\ell)$ is spanned by $\epsilon_2,e_1,e_2$ and $e_3^*$.  Recall that the parabolic subgroup $P(\ell) \subseteq \G_2$ stabilizing $\ell$ has Levi subgroup $\GL_2$.  With this ordered basis of $V$, one computes that $M_H(V)$ acts on $V$ as the subgroup of elements of the form
\begin{equation}\label{eqn:Vsubgrp}\left(\begin{array}{ccccc} 1 & x & * & * &* \\ & 1 & 0& * & * \\& & g & * & * \\ & & & \det(g) & x' \\ &&&& \det(g) \end{array}\right),\end{equation}
for $g \in \GL_2$.  The character $\xi$ on such an element is only a function of $x$ and $x'$.  Thus $K_V(\varphi)$ vanishes in this case by constant term of $\varphi$ along the upper right $4 \times 2$ block of \eqref{eqn:Vsubgrp}.
\end{proof}

We now consider the case in which the isotropic line is not contained in $V_7$, but instead spanned by $\epsilon_1$.  Then $K_V(\varphi)$ again vanishes.  To see this, first note that $\Ann_{R}(\ell)$ is spanned by $e_1^*, e_2^*,e_3^*$ and $\epsilon_2$.  Using these elements to form an ordered basis of $\Ann_{R}(\ell)$ and thus $V$, one finds that $M_H(V)$ acts on $V$ by elements of the form
\begin{equation}\label{eqn:411ep1} R' = \left\{r' = \left(\begin{array}{cccc} 1 & x & * &* \\ & 1 & 0 & x' \\ & & g & * \\ & &&1\end{array}\right) : g \in \SL_3\right\}.\end{equation}
Define a character $\xi'$ on such an element by $\xi'(r') =\psi(x + x')$.  Then the character $\xi$ on $N$ induces the character $\xi'$ on $R'$.

\begin{lemma}
Suppose $\varphi$ is a cusp form on $\GL_6(\BA)$.  Then the integral
\[\int_{R'(F)\backslash R'(\A)}{\xi'(r')\varphi(r')\,dr'}\]
vanishes.  Consequently, $K_V(\varphi) = 0$.
\end{lemma}

\begin{proof}
By changing bases, this integral over $R'$ becomes an integral over the subgroup
\[R'' = \left\{\left(\begin{array}{cccc} g & 0 & 0 &* \\ *& 1 & x& * \\ & & 1& x' \\ & &&1\end{array}\right): g \in \SL_3\right\}.\]
Set
\[R''' = \left\{\left(\begin{array}{cccc} g & 0 & 0 &* \\ & 1 & 0& 0 \\ & & 1& 0 \\ & &&1\end{array}\right): g \in \SL_3\right\}.\]
Here the $*$ is a $3 \times 1$ block. In fact, one has that the integral of a cusp form $\varphi$ over $[R''']$ vanishes, from which the lemma follows.

To see that $\int_{[R''']}{\varphi(r)\,dr} = 0$, one proceeds as follows.  Denote by $N'$ the abelian unipotent group consisting of matrices of the form
\[n(x_1,x_2,x_3) := \left(\begin{array}{cccc} 1_3 & x_1 & x_2 & x_3 \\ & 1 & 0& 0 \\ & & 1& 0 \\ & &&1\end{array}\right).\]
Here $x_1, x_2, x_3$ are $3 \times 1$ column vectors.  Denote by $N''$ the subgroup of $N'$ consisting of those matrices with $x_1  = 0$, and denote by $N'''$ the subgroup of $N'$ consisting of those matrices with $x_1 = x_2 = 0$.

For $r_2 \in F^3$ a column vector, define
\[\varphi_{r_2}(g) = \int_{[N'']}{\psi(r_2 x_2)\varphi(n(0,x_2,x_3)g)\,dx_2\,dx_3}.\]

Then we have
\[\int_{[N''']}{\varphi(ng)\,dn} = \varphi_{0}(g) + \sum_{\gamma \in P'_{2,1}\backslash \SL_3(F)}{\varphi_{(0,0,1)}(\gamma g)}.\]
Here $P'_{2,1}$ is the subgroup of $\SL_3$ consisting of matrices of the form $\mm{g}{*}{0}{1}$ with $g \in \SL_2$ and $*$ a $2 \times 1$ column vector, and $\SL_3$ is embedded in $\GL_6$ as $g \mapsto \mm{g}{0}{0}{1_3}$.

By Fourier expanding $\varphi_{0}(g)$ along the $x_1$ coordinates, and letting $\SL_3$ act on the Fourier coefficients, one verifies immediately that $\int_{[\SL_3]}{\varphi_{0}(g)\,dg} = 0$, using the cuspidality of $\varphi$ along the unipotent radical of the $(3,3)$ and $(2,4)$ parabolics.  Thus, we are reduced to showing the vanishing of
\begin{equation}\label{eqn:Runfold}\int_{[\SL_3]}{\sum_{\gamma \in P'_{2,1}(F)\backslash \SL_3(F)}{\varphi_{(0,0,1)}(\gamma g)\,dg}} = \int_{P'_{2,1}(F)\backslash \SL_3(\A)}{\varphi_{(0,0,1)}(g)\,dg}.\end{equation}

To show the vanishing of the right-hand side of \eqref{eqn:Runfold}, Fourier expand $\varphi_{(0,0,1)}(g)$ along the two-dimensional unipotent subgroup consisting of matrices $n(x_1,0,0)$ where the last entry of $x_1$ is $0$.  Then, the $\SL_2(F)$ in $P'_{2,1}$ acts on these Fourier coefficients with two orbits, corresponding to the constant and nonconstant terms.  Now, one proceeds as above: the integral of the constant term over $[P_{2,1}']$ vanishes by the cuspidality of $\varphi$ along the unipotent radical of the $(2,4)$ parabolic.  The integral of the nonconstant terms vanish by the cuspidality of $\varphi$ along the unipotent radical of the $(1,5)$ parabolic.  This completes the proof of the lemma.\end{proof}

\subsection{The last orbit}\label{ssec:GRperiod}
Finally, we consider the integrals $I_V(\phi,s)$ and $J_V(\phi,s)$ associated to the orbit $V = V(\NullU)$.
We continue assuming that $\Rel(s)$ is large so that Proposition \ref{prop:absConv} holds.
 Recall that in this case, $\NullU = \{(0,\mm{*}{*}{0}{0})\} \subseteq \Theta$ in the notation of the Cayley-Dickson construction of $\Theta$ from $D = M_2(F)$.  Furthermore, the stabilizer $H_0 \cap P$ of $V(\NullU)$ inside $\G_2$ is the Heisenberg parabolic $P(\NullU)$, and we have $H_0 \cap P = (H_0 \cap M) (H_0 \cap U)$ with $H_0\cap M\simeq \GL_2$ being a Levi of $P(\NullU)$. $\GL_2\simeq H_0\cap M$ acts on $\Theta$ via
\begin{equation}\label{eqn:LeviAct}g.(x,y) = \left(gxg^{-1}, \mm{\det(g)}{}{}{1}yg^{-1}\right),\;g\in \GL_2,(x,y)\in \Theta\end{equation}
in the notation of the Cayley-Dickson construction.

\begin{lemma}\label{lem:helpComp} Set $\overline{\NullU} = \{(0,\mm{0}{0}{*}{*})\} \subseteq \Theta$, and define $\overline{V} = \{n(x,y;z): x,y, z \in \overline{\NullU}\}$.
\begin{enumerate}
\item The inclusion $\overline{V} \rightarrow N$ induces an isomorphism $\overline{V} \simeq (P \cap N)\backslash N$.
\item The character $\xi$ is trivial on $H \cap U$ and on $\overline{V}$.
\item Denote by $dv$ the Haar measure on $\overline{V}$.  Then for $x \in H_0\cap M \subseteq P(\Omega) \subseteq \G_2$, $x\bar{V}x^{-1}=\bar{V}$ with the Jacobian $\frac{d(x v x^{-1})}{dv} = |\det(x)|^{-3}=\delta_{P(\Omega)}(x)^{-1}$.
\end{enumerate}
\end{lemma}
\begin{proof} The stabilizer of $V(\NullU)$ inside of $N$ was computed in Lemma \ref{lem:222Stab} to be the elements $n(x,y;z)$ with $x,y,z \in \NullU^\perp \subseteq \Theta$.  With our choice of $\NullU$, $\NullU^\perp = \{(x,y): x \in M_2(F), y \in \mm{*}{*}{0}{0}\}$.  Thus, $H \cap N = \{n(x,y;z): x,y,z \in \NullU^{\perp}\}$.  The first item follows from this.

For the second item, again from Lemma \ref{lem:222Stab}, one has that the elements of $N$ that act trivially on $V(\NullU)$ are the $n(x,y;z)$ with $x,y,z \in \NullU$.  Thus $H \cap U = \{n(x,y;z): x,y,z \in \NullU\}$, and the second item is clear.

Finally, the third item follows immediately from \eqref{eqn:LeviAct}, \eqref{eqn:gconjN} and the last part of the paragraph \ref{ssec:rho}.
\end{proof}

Denote $I_1(\phi,s) = I_V(\phi, s)$ and $I_2(\phi, s) = J_V(\phi,s)$. Let
\begin{equation*}\label{eq:psii}
\psi_1(x) = (1-\htau_P(H_P(x)-T))\phi(x)e^{\langle s\varpi, H_P(x) \rangle}, \quad
\psi_2(x) = \htau_P(H_P(x)-T)M(s)\phi(x)e^{\langle -s\varpi, H_P(x) \rangle}.
\end{equation*}

As we discussed before, we have $H\cap P=(H_0\cap P)\times (N\cap P)$, $N=(N\cap P)\times \bar{V}$, and $H_0\cap P=(H_0\cap M)\times (H_0\cap N)$ with $\GL_2\simeq H_0\cap M$. Moreover, since both $P$ and $Q$ are semistandard parabolic subgroup and $M$ is a semistandard Levi subgroup, we also have $N\cap P=(N\cap M)\times (N\cap U)$. Combining with Lemma \ref{lem:helpComp}, for $i=1,2$, we have
\begin{eqnarray*}
I_i(\phi,s)&=&\int_{(H\cap P)(F)\back H(\BA)} \psi_i(h)\xi(h)dh  \\
&=&\int_{(H_0\cap P)(F)\back H_0(\BA)} \int_{(N\cap P)(F)\back N(\BA)} \psi_i(nh_0) \xi(n)dndh_0 \\
&=&\int_{K_{\G_2}} \int_{[H_0\cap M]} \int_{[H_0\cap U]} \int_{\bar{V}(\BA)} \int_{[N\cap M]} \int_{[N\cap U]} \psi_i(n_u n_m v u_h h k)
\xi(n_m) \delta_{P(\Omega)}(h)^{-1} \,d\mu \\
&=& \int_{K_{\G_2}} \int_{[H_0\cap M]}  \int_{\bar{V}(\BA)} \int_{[N\cap M]}  \psi_i(n_m v h k)\xi(n_m) \delta_{P(\Omega)}(h)^{-1} dn_m dv dh dk \\
&=&\int_{K_{\G_2}} \int_{\bar{V}(\BA)} \int_{[H_0\cap M]}  \int_{[N\cap M]}  \psi_i(n_m h v k)\xi(n_m) \delta_{P(\Omega)}(h)^{-2}  dn_m dh dv dk.
\end{eqnarray*}
Here $d \mu = dn_u dn_m dv du_h dh dk$.

\begin{lemma}\label{lem:VGR}
Under the isomorphism $M\simeq \{(g,\lambda)|\; g\in \GL_6,\lambda\in \GL_1,\;\det(g)=\lambda^3\}$ specified in the proof below, we have the following facts:
\begin{enumerate}
\item $H_0\cap M\simeq \{(diag(h,h,h),\det(h))\in \GL_6\times \GL_1|\; h\in \GL_2\}$;
\item $N\cap M\simeq \left\{\left(\begin{pmatrix}I_2&x&z\\0&I_2&y\\0&0&I_2\end{pmatrix},1\right)|\; x,y,z\in M_2\right\}$;
\item the character $\xi$ of $N\cap M$ is given by $\xi\left(\begin{pmatrix}I_2&x&z\\0&I_2&y\\0&0&I_2\end{pmatrix}\right)=\psi(\tr(x)+\tr(y))$.
\end{enumerate}
In particular, up to modulo the center, the integral $\int_{[H_0\cap M]}  \int_{[N\cap M]}$ gives us the period integral of the Ginzburg-Rallis model defined in the introduction.
\end{lemma}

\begin{proof} We first chose a basis of $V$.  To do so, fix the ordered basis $b_1 = (0,\mm{1}{0}{0}{0})$ and $b_2 = (0,\mm{0}{1}{0}{0})$ of $\NullU$.  We have
\[V=V(\NullU) = V(\NullU)(a_1) \oplus V(\NullU)(a_2) \oplus V(\NullU)(a_3) = \NullU \oplus \NullU \oplus \NullU.\]
This basis of $b_1, b_2$ of $\NullU$ thus gives an ordered basis $b_1^{(1)}, b_2^{(1)}, b_1^{(2)}, b_2^{(2)}, b_1^{(3)},b_2^{(3)}$ of $V(\NullU)$, with $b_j^{(k)}$ meaning the element $b_j$ considered in $V(\NullU)(a_k) = \NullU$.

The map $M \rightarrow \GL(V)\simeq \GL_6$ comes from the action of $M$ on $V=V(\Omega)$, together with the identification $\GL(V) \simeq \GL_6$ given by the above ordered basis of $V$.  More precisely, suppose an element $g \in M$ acts on the right of $V$ by an element $m = m(g)$ of $\GL(V)$.  Then, under the identification of $M$ with $\{(\lambda,h) \in \GL_1 \times \GL_6: \lambda^3 = \det(h)\}$ the image of $g$ in $\GL_6$ is $\,^*m^{-1}$; see \eqref{eqn:MactJ}.

With these identifications, the first item of the lemma is clear using \eqref{eqn:LeviAct}.  We now prove the second and third items of the lemma.

Thus suppose $x = (w_x,*), y = (w_y, *), z = (w_z,*) \in \NullU^\perp \subseteq \Theta$.  Via Lemma \ref{HeisTransform}, one computes that in the above ordered basis of $V$, $n(x,y;z)$ acts on $V$ via the matrix
\[\tilde{w}(x,y,z) = \left(\begin{array}{ccc} 1_2 & 0 & 0 \\ -w_x^*& 1_2 & 0 \\ * & -w_y^* & 1_2 \end{array}\right)\]
and that the image of $H_{V} \cap N$ inside $M_H(V)$ contains all matrices in the unipotent radical of the lower-triangular $(2,2,2)$ parabolic of $\GL(V)$. Here $w \mapsto w^*$ is the conjugation on the quaternion algebra $D = M_2(F)$.

Thus the image of $n(x,y;z) \in H \cap P$ in $\GL_6$ is
\[w(x,y,z) = \left(\begin{array}{ccc} 1_2 & w_x & * \\ 0 & 1_2 & w_y\\ 0 & 0 & 1_2 \end{array}\right).\]
Since $\xi(n(x,y;z)) = \psi(\tr(x+y)) = \psi(\tr(w_x+w_y))$, the lemma follows.\end{proof}

By the lemma above, we have
\begin{eqnarray}\label{eq:IiSemiDone}
I_i(\phi,s)&=&\int_{K_{\G_2}} \int_{\bar{V}(\BA)} \int_{[H_0\cap M]}  \int_{[N\cap M]}  \psi_i(n_m h v k)\xi(n_m) \delta_{P(\Omega)}(h)^{-2}  dn_m dh dv dk\\
&=&\int_{K_{\G_2}} \int_{\bar{V}(\BA)} \int_{[H_0\cap M]^1}  \int_{[N\cap M]}\int_{A_{P}^{\infty}}  \psi_i(an_m h v k)\xi(n_m) \delta_{P(\Omega)}(a)^{-2} da dn_m dh dv dk. \nonumber
\end{eqnarray}

\begin{proposition}\label{prop:truncIdentity}
For all $T \in \R$ sufficiently big we have the following equality of meromorphic functions on $\C$
\begin{align*}
\calP_{\G_2}(\La^{T}E(\phi, s))  &=
\dfrac{e^{(s-1/2)T}}{s-1/2}
\int_{K_{\G_2}}\int_{\overline{V}(\A)} \int_{[H_0\cap M]^1}  \int_{[N\cap M]} e^{\langle \frac{1}{2}\varpi, H_{P}(vk) \rangle} \phi( n_m h v k )\xi(n_m) \,dn_m dh dv dk  \quad + \\
& \dfrac{e^{(-s-1/2)T}}{-s-1/2}\int_{K_{\G_2}}\int_{\overline{V}(\A)}
\int_{[H_0\cap M]^1}  \int_{[N\cap M]} e^{\langle \frac{1}{2}\varpi, H_{P}(vk) \rangle} M(s)\phi(n_m h v k ) \xi(n_m) \, dn_m dh dv dk.
\end{align*}
\end{proposition}

\begin{proof}
It is enough to prove the equality when $Re(s)$ is large. Let $\Rel(s)$ be large.
Using the decomposition \eqref{eq:periodDecomp} and the vanishing results of section
\ref{ssec:vanish1} and \ref{ssec:vanish2}
we get $\calP_{\G_2}(\La^{T}E(\phi, s)) = I_1(\phi,s) + I_2(\phi,s)$.
We first study the term $I_1(\phi,s)$. Using the expression \eqref{eq:IiSemiDone} and the definition of the function $\psi_1$, we have
\begin{eqnarray*}
I_1(\phi,s)&=&\int_{K_{\G_2}} \int_{\bar{V}(\BA)} \int_{[H_0\cap M]^1}  \int_{[N\cap M]}\int_{A_{P}^{\infty}} (1-\hat{\tau}_P(H_P(avk)-T))\\ &&e^{\langle s\varpi, H_{P}(avk) \rangle} \delta_P(a)^{1/2}\delta_{P(\Omega)}(a)^{-2} \phi(n_m h v k)\xi(n_m)  da dn_m dh dv dk.
\end{eqnarray*}
By Lemma \ref{lem:VGR} together with the computation of the modular characters in section \ref{ssec:rho}, we have
$$\delta_P(a)=e^{\langle 11\varpi, H_{P}(a) \rangle},\;\delta_{P(\Omega)}(a)=e^{\langle 3\varpi, H_{P}(a) \rangle},\; a\in A_{P}^{\infty}.$$
Therefore, we get
\begin{eqnarray*}
I_1(\phi,s)&=&\int_{K_{\G_2}} \int_{\bar{V}(\BA)} \int_{[H_0\cap M]^1}  \int_{[N\cap M]}\int_{A_{P}^{\infty}} (1-\hat{\tau}_P(H_P(avk)-T))\\ &&e^{\langle s\varpi, H_{P}(avk) \rangle -\langle \frac{1}{2}\varpi, H_{P}(a) \rangle}  \phi(n_m h v k)\xi(n_m)  da dn_m dh dv dk.
\end{eqnarray*}
We can change the integral on $A_{P}^{\infty}$ to the integral on $\ago_P$, this implies that
\begin{eqnarray*}
I_1(\phi,s)&=&\int_{K_{\G_2}} \int_{\bar{V}(\BA)} \int_{[H_0\cap M]^1}  \int_{[N\cap M]}\int_{\ago_P} (1-\hat{\tau}_P(X+H_P(vk)-T))\\
&&e^{\langle s\varpi, X+H_{P}(vk) \rangle -\langle \frac{1}{2}\varpi, X \rangle}  \phi(n_m h v k)\xi(n_m)  dX dn_m dh dv dk\\
&=&\int_{K_{\G_2}} \int_{\bar{V}(\BA)} \int_{[H_0\cap M]^1}  \int_{[N\cap M]}\int_{\ago_P} (1-\hat{\tau}_P(X-T)) e^{\langle (s-\frac{1}{2})\varpi, X \rangle}\\
&&e^{\langle \frac{1}{2}\varpi, H_{P}(vk) \rangle} \phi(n_m h v k)\xi(n_m)  dX dn_m dh dv dk
\end{eqnarray*}
where the second equality is by the change of variables $X + H_P(vk) \mapsto X$. The integral on $\ago_P$ yields
$$\int_{\ago_P} (1-\hat{\tau}_P(X-T)) e^{\langle (s-\frac{1}{2})\varpi, X \rangle}dX=\int_{-\infty}^{T} e^{(s-\frac{1}{2})X}dX=\dfrac{e^{(s-1/2)T}}{s-1/2}.$$
This implies that
$$I_1(\phi,s)=\dfrac{e^{(s-1/2)T}}{s-1/2}
\int_{K_{\G_2}}\int_{\overline{V}(\A)} \int_{[H_0\cap M]^1}  \int_{[N\cap M]} e^{\langle \frac{1}{2}\varpi, H_{P}(vk) \rangle} \phi( n_m h v k )\xi(n_m) \,dn_m dh dv dk.$$
Similarly, we can also show that
$$I_2(\phi,s)=\dfrac{e^{(-s-1/2)T}}{-s-1/2}\int_{K_{\G_2}}\int_{\overline{V}(\A)}
\int_{[H_0\cap M]^1}  \int_{[N\cap M]} e^{\langle \frac{1}{2}\varpi, H_{P}(vk) \rangle} M(s)\phi(n_m h v k ) \xi(n_m) \, dn_m dh dv dk.$$
This finishes the proof of the proposition.
\end{proof}

\subsection{The proof of Theorem \ref{main}}\label{ssec:proofMain}
We are ready to prove Theorem \ref{main}. Taking the residue at $s=1/2$ of
the equality in Proposition \ref{prop:truncIdentity} we get
\begin{eqnarray*}
&&\int_{K_{\G_2}}\int_{\overline{V}(\A)} \int_{[H_0\cap M]^1}  \int_{[N\cap M]} e^{\langle \frac{1}{2}\varpi, H_{P}(vk) \rangle} \phi( n_m h v k )\xi(n_m) \,dn_m dh dv dk=\calP_{\G_2}(\La^{T}Res_{s=1/2}E(\phi, s))\\
&&  + e^{-T}\int_{K_{\G_2}}\int_{\overline{V}(\A)}\int_{[H_0\cap M]^1}  \int_{[N\cap M]}
e^{\langle \frac{1}{2}\varpi, H_{P}(vk) \rangle} Res_{s=1/2}M(s)\phi(n_m h v k ) \xi(n_m) \, dn_m dh dv dk.
\end{eqnarray*}
The interchange of the residue with truncation can be justified as in \cite{arthur2}, pages 47-48,
given that the period is absolutely convergent. By Lemma \ref{lem:VGR}, the inner integral $\int_{[H_0\cap M]^1}  \int_{[N\cap M]}$ is the Ginzburg-Rallis period. Hence if the Ginzburg-Rallis period is nonzero on the space of the representation $\pi$, one can always
find a $\phi \in \calA_{\pi}$ such that the integral
$$\int_{K_{\G_2}}\int_{\overline{V}(\A)} \int_{[H_0\cap M]^1}  \int_{[N\cap M]} e^{\langle \frac{1}{2}\varpi, H_{P}(vk) \rangle} \phi( n_m h v k )\xi(n_m) \,dn_m dh dv dk$$
does not vanish. This is a standard argument that can be proved as in \cite{ichYam}, Lemma 5.8. As a result we can always find $\phi\in \calA_{\pi}$ such that the left hand side (hence the right hand side) of the equality above is nonzero.
Since $Res_{s=1/2} E(\phi,s) \neq 0$ if and only if $Res_{s=1/2}M(s)\neq 0$ we must have that the former is non-zero for some $\phi$.
This completes the proof of Theorem \ref{main} under the assumption of Proposition \ref{prop:absConv}.

\begin{remark}
We expect that the period $\calP_{\G_2}(Res_{s=1/2}E(\phi, s))$ is absolutely convergent.
To put it differently, the period $\calP_{\G_2}$ is absolutely convergent on the space
of the square integrable representation $\Pi$ generated by $Res_{s=1/2}E(\phi, s)$.
It would then easily follow from
the identity above that $\calP_{\G_2}|_{\Pi} \neq 0$ if and only if $\calP_{GR}|_{\pi} \neq 0$.
Since we do not need this result to prove the relation between $\calP_{GR}$ and the central value
of the exterior cube $L$-function of $\GL_6$ we didn't pursue this goal here.
\end{remark}

\section{The proof of Proposition \ref{prop:absConv}}\label{sec:convergence}
In this section, we will prove that the integrals associated to each orbits are absolutely convergent. We first reduce the proof to some statement of norms on adelic variety. Then we introduce a notion of good pair and we show that it is enough to prove that $(H,H_V)$ is a good pair for all $V \in \calV/H$. Finally, we will show that $(H,H_V)$ is a good pair for all $V \in \calV/H$.

\subsection{Notations}\label{ssec:notNorm}
If $f_1$ and $f_2$ are two positive functions on a set $X$, we write
$$f_1\ll f_2$$
if there exists $C,d>0$ such that $f_1(x)\leq Cf_2(x)^d$ for all $x\in X$. We write
$$f_1\sim f_2$$
if $f_1\ll f_2$ and $f_2\ll f_1$.

If $X$ is an algebraic variety defined over $F$, we use $||\cdot ||_{X}$ to be the norm on $X(\BA_{\bar{F}})$. Here $X(\BA_{\bar{F}})=\cup_{E} X(\BA_{E})$ where $E$ runs over all the finite extension of $F$. We refer the readers to Appendix A of \cite{B17} for the definition and basic properties of the norms on adelic varieties.

\subsection{Some reduction}
Let $\| \cdot \|_G$ be the norm on $G(\A)$. We first study the majorization of constant terms.
\begin{lem}\label{lem:normBound}
For all $n \in \N$, there exist $c \in \R$ and $T_0\in \BR$ large such that for all $s \in \C$ and $T\in \BR$ with $Re(s) > c$ and $T>T_0$, we have:
\[
  (1-\htau_{P}(H_{P}(x) - T))\phi(x)e^{\langle s\varpi, H_{P}(x) \rangle} \leq
  (\min_{\delta \in P(F)}\|\delta x\|)^{-n}
\]
and
\[
  \htau_{P}(H_{P}(x) - T)M(s)\phi(x) e^{\langle -s\varpi, H_{P}(x) \rangle}\leq
  (\min_{\delta \in P(F)}\|\delta x\|)^{-n}
\]
for all $x \in G(\A)$.
\end{lem}

\begin{proof}
This follows from the fact that cusp forms are rapidly decreasing.
\end{proof}

\begin{prop}
For all $V \in \calV/H$, there exists $D>0$ such that the integral
\begin{equation}\label{period integral}
\int_{H_V(F)\back H(\BA)} (\inf_{\delta\in H_V(F)} ||\delta h||_{H})^{-d}  dh
\end{equation}
is absolutely convergent for all $d>D$.
\end{prop}

\begin{proof}
Since
\begin{eqnarray*}
\int_{H_V(F)\back H(\BA)} (\inf_{\delta\in H_V(F)} ||\delta h||_{H})^{-d}  dh\leq \int_{H_V(F)\back H(\BA)} \sum_{\delta \in H_V(F)} ||\delta h||_{H}^{-d} dh=\int_{H(\BA)} ||h||_{H}^{-d}dh,
\end{eqnarray*}
we only need to show that the integral
$$\int_{H(\BA)} ||h||_{H}^{-d}dh$$
is absolutely convergent for all $d>D$. This just follows from Proposition A.1.1(vi) of \cite{B17}.
\end{proof}

Combining the lemma and the proposition above, in order to prove Proposition \ref{prop:absConv}, it is enough to prove the following proposition.
\begin{prop}\label{norm}
For all $V \in \calV/H$, we have
$$\inf_{\delta\in P(F)}||\delta \gamma_V h||_G \sim  \inf_{\delta\in H_V(F)} || \gamma_V \delta h||_H$$
for all $h\in H(\BA)$. This is equivalent to show that for all $V \in \calV/H$, we have
$$\inf_{\delta\in P(V)(F)}||\delta  h||_G \sim  \inf_{\delta\in H_V(F)} || \delta h||_H$$
for all $h\in H(\BA)$. Recall that $P(V)=\gamma_{V}^{-1}P\gamma_V$ is the parabolic subgroup of $\E_6$ stabilizing $V$ and $H_V=H\cap P(V)$.
\end{prop}
The goal of the rest part of this section is to prove Proposition \ref{norm}.

\subsection{Some abstract theory}
In this subsection, let $G$ be a linear algebraic group defined over $F$, $H$ and $P$ are two closed subgroups of $G$ also defined over $F$ with $H_P=H\cap P$. Assume that $G/H$ is quasi-affine. We use $i$ to denote the natural embedding $P/H_P\rightarrow G/H$. We say $(H,H_P)$ is a good pair if
\begin{equation}\label{good pair}
\inf_{\gamma\in H_P(F)}||\gamma h||_H \ll \inf_{\gamma\in H_P(\bar{F})}||\gamma h||_H
\end{equation}
for all $h\in H(\BA)$.

\begin{lem}\label{quasi-affine}
$i(P/H_P)$ is open in its closure $\overline{i(P/H_P)}$. In particular, this implies that $P/H_P$ is quasi-affine since we have assumed that $G/H$ is quasi-affine.
\end{lem}

\begin{proof}
By Theorem 1.9.5 of \cite{S98}, $i(P/H_P)$ contains an open subset $U$ of its closure $\overline{i(P/H_P)}$. This implies that $i(P/H_P)=\bigcup_{p\in P} p\cdot U$ is open in $\overline{i(P/H_P)}$.
\end{proof}

\begin{prop}\label{norm descent}
Assume that $(H,H_P)$ is a good pair. Then for all $h\in H(\BA)$, we have
$$\inf_{\gamma \in P(F)} ||\gamma h||_G \sim \inf_{\gamma\in H_P(F)} ||\gamma h||_H.$$
\end{prop}

\begin{proof}
One direction is trivial, we only need to show that
$$\inf_{\gamma\in P(F)}||\gamma h||_G \gg \inf_{\gamma\in H_P(F)} ||\gamma h||_H.$$
Since $P/H_P$ is quasi-affine, by the proof of Proposition A.1.1(ix) of \cite{B17}, there exists a set theoretic section $s:(P/H_P)(\bar{F}) \rightarrow P(\bar{F})$ of the projection map $pr_1:P\rightarrow P/H_P$ such that
$$||s(x)||_G\ll ||x||_{P/H_P}$$
for all $x\in (P/H_P)(\bar{F})$. We will use $pr_2$ to denote the projection map $G\rightarrow G/H$ and use $i$ to denote the natural embedding $P/H_P\rightarrow G/H$. Then for all $\gamma\in P(F)$, we have $i(pr_1(\gamma))=pr_2(\gamma)$. By Proposition A.1.1(ii), (iv) of \cite{B17} and Lemma \ref{quasi-affine} above, we have
\begin{equation}\label{norm descent 1}
||\delta||_{P/H_P}\sim ||i(\delta)||_{G/H}
\end{equation}
for all $\delta\in P/H_P(\bar{F})$. Then by applying Proposition A.1.1(ii) of \cite{B17} again, we have
$$||pr_1(\gamma)||_{P/H_P} \sim ||i(pr_1(\gamma))||_{G/H}=||pr_2(\gamma)||_{G/H} \ll ||\gamma h||_G$$
for all $\gamma\in P(F)$ and $h\in H(\BA)$. This implies that
$$\inf_{\gamma'\in H_P(\bar{F})}||\gamma'h||_G\leq ||s(pr_1(\gamma))^{-1}\gamma h||_G\ll ||s(pr_1(\gamma))||_G ||\gamma h||_G\ll ||pr_1(\gamma)||_{P/H_P} ||\gamma h||_H\ll ||\gamma h||_H$$
for all $(\gamma,h)\in P(F)\times H(\BA)$. If we take $\inf$ over $\gamma$, we get
$$\inf_{\gamma\in P(F)}||\gamma h||_G \gg \inf_{\gamma\in H_P(\bar{F})} ||\gamma h||_H.$$
Since $(H,H_P)$ is a good pair, we have
$$\inf_{\gamma\in P(F)}||\gamma h||_G \gg \inf_{\gamma\in H_P(F)} ||\gamma h||_H.$$
This proves the proposition.
\end{proof}

\begin{rmk}\label{norm descent remark}
If we only want to prove
$$\inf_{\gamma\in P(\bar{F})}||\gamma h||_G \sim \inf_{\gamma\in H_P(\bar{F})} ||\gamma h||_H$$
for all $h\in H(\BA_{\bar{F}})$, we don't need to assume that $(H,H_P)$ is a good pair.
\end{rmk}

The following corollary is a direct consequence of the proposition above, it reduces the proof of Proposition \ref{norm} to the proof of some good pair arguments.

\begin{cor}\label{reduction to good pair}
In order to prove Proposition \ref{norm}, it is enough to show that $(H,H_V)$ are good pairs for all $V \in \calV/H$.
\end{cor}

For the rest part of this subsection, we will discuss some abstract theory for good pairs.

\begin{prop}\label{good pair prop}
\begin{enumerate}
\item $(H,H)$ is a good pair.
\item If there exists a closed subvariety $H'$ of $H$ defined over $F$ such that the morphism
$$H_P\times H'\rightarrow H:\; (h_1,h_2)\mapsto h_1h_2 $$
is an isomorphism, then $(H,H_P)$ is a good pair.
\item If $H=H_0\ltimes U$ with $H_0$ reductive and $U$ unipotent, and $H_P=H_{0,P}\ltimes U_P$ with $H_{0,P}=H_0\cap H_P$ and $U_P=U\cap H_P$. Assume that there exists a closed subgroup $H_1$ of $H_0$ such that $H_0=H_{0,P}\times H_1$, then $(H,H_P)$ is a good pair.
\item If $H=H_0\ltimes U$ with $H_0$ reductive and $U$ unipotent. Assume that there exists a parabolic subgroup $Q=L\ltimes N$ of $H_0$ such that $H_{P}\subset Q\ltimes U$ and $(Q\ltimes U,H_P)$ is a good pair, then $(H,H_P)$ is a good pair.
\end{enumerate}
\end{prop}

\begin{proof}
(1) has been proved in Proposition A.1.1(ix) of \cite{B17}. For (2), let $h\in H(\BA)$. Then we can write $h$ as $h_1h_2$ with $h_1\in H_P(\BA)$ and $h_2\in H'(\BA)$. We have
$$||h||_H\sim ||h_1||_{H_P} ||h_2||_{H'}.$$
Then for all $\gamma\in H_P(\bar{F})$, we have
$$||\gamma h||_H\sim ||\gamma h_1||_{H_P} ||h_2||_{H'},$$
which implies that
$$\inf_{\gamma\in H_P(\bar{F})}||\gamma h||_H\sim \inf_{\gamma\in H_P(\bar{F})}||\gamma h_1||_{H_P} ||h_2||_{H'}\gg \inf_{\gamma\in H_P(F)}||\gamma h_1||_{H_P} ||h_2||_{H'}$$
$$\sim \inf_{\gamma\in H_P(F)}||\gamma h_1h_2||_H=\inf_{\gamma\in H_P(F)}||\gamma h||_H.$$
Here $\inf_{\gamma\in H_P(\bar{F})}||\gamma h_1||_{H_P} \gg \inf_{\gamma\in H_P(F)}||\gamma h_1||_{H_P}$ follows from (1). This proves (2).
\\
\\
(3) is a direct consequence of (2). For (4), we just need to use the Iwasawa decomposition $H(\BA)=Q(\BA)U(\BA)K$ together with the fact that elements of $K$ will not change the norm.
\end{proof}

\begin{cor}\label{good pair cor}
Let $H=H_0\ltimes U$ with $H_0$ reductive and $U$ unipotent. Assume that $H_P=H_{P,0}\ltimes U_P$ with $H_{P,0}=H_0\cap H_P$ and $U_P=U\cap H_P$. Moreover, assume that there exists a parabolic subgroup $Q=L\ltimes N$ of $H_0$ such that $H_{P,0}\subset Q$ and $H_{P,0}=L_P\ltimes N_P$ with $L_P=L\cap H_{P,0}$ and $N_P=N\cap H_{P,0}$. Finally, we assume that there exists a closed subgroup $L_1$ of $L$ such that $L=L_P\times L_1$. Then $(H,H_P)$ is a good pair.
\end{cor}

\begin{proof}
By the Iwasawa decomposition, it is enough to show that $(L\ltimes (N\ltimes U),H_P)=(L\ltimes (N\ltimes U),L_P\ltimes (N_P\ltimes U_P))$ is a good pair. This follows from Proposition \ref{good pair prop}(3).
\end{proof}

\subsection{The good pair $(\G_2,\SL_3)$}
In this subsection, $G=G_0\ltimes N$ where $G_0=\G_2$ and $N$ is some unipotent group. Let $H=H_0\ltimes N'$ be a subgroup of $G$ with $N'$ being a subgroup of $N$ and $H_0=\SL_3\subset \G_2=G_0$. Our goal is to prove the following proposition.

\begin{prop}\label{G_2 good pair proposition}
The pair $(G,H)$ is a good pair.
\end{prop}

We need some preparation. First we recall the embedding of $\SL_3$ into $\G_2$. The positive roots of $\G_2$ is $\{\alpha,\beta,\alpha+\beta,\alpha+2\beta,\alpha+3\beta,2\alpha+3\beta\}$ with $\alpha,\beta$ be the two simple roots. The positive roots of $\SL_3$ is $\alpha_0,\beta_0,\alpha_0+\beta_0$. We embed $\SL_3$ into $\G_2$ via
$$\alpha_0\rightarrow \alpha,\;\beta_0\rightarrow \alpha+3\beta,\;\alpha_0+\beta_0\rightarrow 2\alpha+3\beta.$$

Let $P_{\alpha}=M_{\alpha}U_{\alpha}$ be the maximal parabolic subgroup of $\G_2$ associated to $\alpha$. Then $M_{\alpha}\simeq \GL_2$ and $U_{\alpha}$ contains the roots $\{\beta,\alpha+\beta,\alpha+2\beta,\alpha+3\beta,2\alpha+3\beta\}$. Let $U_0$ be the center of $U_{\alpha}$ which is generated by $\{\alpha+3\beta,2\alpha+3\beta\}$. Then we can write $U_{\alpha}$ as $U_0U_1$ where $U_1$ is a closed subvariety of $U_{\alpha}$ generated by $\{\beta,\alpha+\beta,\alpha+2\beta\}.$ \textbf{Note that $U_1$ is not a group!} Moreover, we know that $M_{\alpha}U_0$ is a parabolic subgroup of $\SL_3$.

\begin{lem}\label{U1 part}
For all $h\in \SL_3(\BA_{\bar{F}})$ and $u_1\in U_1(\BA)$, we have
$$||hu_1||_G\sim  ||h||_G ||u_1||_G.$$
\end{lem}

\begin{proof}
By the Iwasawa decomposition, it is enough to consider the case when $h=mu_0$ with $m\in M_{\alpha}(\BA_{\bar{F}})$ and $u_0\in U_0(\BA_{\bar{F}})$. Since $U_{\alpha}=U_0U_1$, $P_{\alpha}=M_{\alpha}U_{\alpha}$ is a parabolic subgroup of $G_0$ and $M_{\alpha}U_0$ is a parabolic subgroup of $H_0$, we have
$$||hu_1||_G=||mu_0u_1||_G\sim ||m_0||_G ||u_0u_1||_G\sim ||m_0||_G ||u_0||_G ||u_1||_G\sim ||m_0u_0||_G ||u_1||_G=||h||_G ||u_1||_G.$$
This proves the lemma.
\end{proof}

Now we are ready to prove Proposition \ref{G_2 good pair proposition}. For $g\in G(\BA)$, we want to show that
$$\inf_{\gamma\in H(\bar{F})} ||\gamma g||_G\gg \inf_{\gamma\in H(F)} ||\gamma g||_G.$$
By the Iwasawa decomposition, it is enough to consider the case when $g=nmu_0u_1$ with $n\in N(\BA),m\in M_{\alpha}(\BA), u_0\in U_0(\BA)$ and $u_1\in U_1(\BA)$. Since $M_{\alpha}U_0\in \SL_3=H_0$, we can write $g$ as $nh_0u_1$ with $n\in N(\BA),h_0\in H_0(\BA)$ and $u_1\in U_1(\BA)$. In order to prove Proposition \ref{G_2 good pair proposition}, it is enough to prove the following proposition.

\begin{prop}
For all $n\in N(\BA),h_0\in H_0(\BA)$ and $u_1\in U_1(\BA)$, we have
\begin{equation}\label{G_2 1}
\inf_{\gamma \in H_0(\bar{F}),\nu \in N'(\bar{F})} ||\nu\gamma nh_0u_1||_G \gg
\inf_{\gamma \in H_0(F),\nu\in N'(F)} ||\nu\gamma nh_0u_1||_G.
\end{equation}
\end{prop}

\begin{proof}
For all $\gamma \in H_0(\bar{F})$ and $\nu\in N'(\bar{F})$, we have
$$||\nu\gamma nh_0u_1||_G = ||(\nu \gamma n\gamma^{-1})\cdot (\gamma h_0) u_1||_G \gg ||(\gamma h_0) u_1||_G.$$
Combine with Lemma \ref{U1 part}, we have
$$||\nu\gamma nh_0u_1||_G \gg ||u_1||_G.$$
This implies that
\begin{equation}\label{G_2 2}
\inf_{\gamma\in H_0(\bar{F}),\nu\in N'(\bar{F})} ||\nu\gamma nh_0u_1||_G \gg ||u_1||_G
\end{equation}
for all $n\in N(\BA),h_0\in H_0(\BA)$ and $u_1\in U_1(\BA)$.

From the properties of norms it follows that there exists $N_0, C_0 > 0$ such that for all $g_1, g_2 \in G(\A_{\bar F})$ we have
\begin{equation}\label{eq:g21}
C_0^{-1}\|g_1\|^{1/N_0}\|g_2\|^{-1} \le \|g_1 g_2\| \le C_0 \|g_1\|^{N_0} \|g_2\|^{N_0}.
\end{equation}
Fix $N_1$ and $C_1$ sufficiently larger than $N_0$ and $C_0$ respectively.
For a given $n\in N(\BA),h_0\in H_0(\BA)$ and $u_1\in U_1(\BA)$, if
$$\inf_{\gamma\in H_0(F),\nu\in N'(F)} ||\nu\gamma nh_0||_G\leq C_1||u_1||_{G}^{N_1}$$
then \eqref{G_2 1} follows from \eqref{G_2 2}. If not, then we have
\begin{equation}\label{G_2 3}
\inf_{\gamma\in H_0(F),\nu\in N'(F)} ||\nu\gamma nh_0||_G\geq C_1||u_1||_{G}^{N_1}.
\end{equation}
Since $(H_0\ltimes N, H_0\ltimes N')$ is a good pair by Proposition \ref{good pair prop}(3),
we also have
$$\inf_{\gamma\in H_0(\bar{F}),\nu\in N'(\bar{F})} ||\nu\gamma nh_0||_G\geq C_1'||u_1||_{G}^{N_1'}$$
for some $C_1',N_1'>0$ large. This and \eqref{eq:g21} imply that
\begin{equation}\label{G_2 4}
\inf_{\gamma\in H_0(\bar{F}),\nu\in N'(\bar{F})} ||\nu\gamma nh_0||_G \le C_3 \inf_{\gamma\in H_0(\bar{F}),\nu\in N'(\bar{F})} ||\nu\gamma nh_0u_1||_{G}^{N_3}
\end{equation}
for some $C_3, N_3 \ge 0$. On the other hand, using \eqref{eq:g21} and \eqref{G_2 3} we have
\begin{equation}\label{G_2 5}
C_4\inf_{\gamma\in H_0(F),\nu\in N'(F)} ||\nu\gamma nh_0u_1||_{G}^{N_4}\leq \inf_{\gamma\in H_0(F),\nu\in N'(F)} ||\nu\gamma nh_0||_G
\end{equation}
for some $C_4, N_4 \ge 0$. Here $C_3,C_4,N_3,N_4$ only depend on $C_0,C_1,C_1',N_0,N_1,N_1'$, not on $n,h_0,u_1$. Invoking once more the fact that
$(H_0\ltimes N, H_0\ltimes N')$ is a good pair,
the inequality  \eqref{G_2 1} follows from \eqref{G_2 4} and \eqref{G_2 5}. This finishes the proof of the proposition and hence the proof of Proposition \ref{G_2 good pair proposition}.
\end{proof}

\subsection{The proof of Proposition \ref{norm}}
In this subsection, we finish the proof of Proposition \ref{norm} and hence the proof of Proposition \ref{prop:absConv}. We go back to the usual notations: $G=\E_6, H=H_0\ltimes N$ with $H_0=\G_2$,
$\{\gamma_V\}_{V \in \calV/H}$ are representatives of the double coset $P\back G/H$. For $V \in \calV/H$, we have defined $P(V)=\gamma_{V}^{-1}P\gamma_V$ and $H_V=H\cap P(V)$. By Corollary \ref{reduction to good pair}, in order to prove Proposition \ref{norm}, it is enough to prove the following proposition.

\begin{prop}\label{prop:MainConv}
For all $V \in \calV/H$, $(H,H_V)$ is a good pair.
\end{prop}

\begin{proof}
For each $V\in \calV/H$, by the computation in Section \ref{sec:orbits}, we have $H_V=H_{0,V}\ltimes N_V$ with $H_{0,V}=H_0\cap P(V)$ and $N_V=N\cap P(V)$.
The groups $H_{0,V}$ are described in Proposition \ref{prop:G2Norbits}.
We know that for $10$ out of the $17$ orbits, $H_{0,V}$ is contained in a maximal parabolic subgroup of $H_0=\G_2$ and $H_{0,V}$ contains the Levi subgroup of this parabolic subgroup (which is isomorphic to $\GL_2$). Hence we know that $(H,H_V)$ is a good pair for these $10$ orbits by Corollary \ref{good pair cor}. For the other $6$ orbits, $H_{0,V}$ is isomorphic to $\SL_3$. Then we know that $(H,H_V)$ is a good pair for these $6$ orbits by Proposition \ref{G_2 good pair proposition}. Finally, for the last orbit, $H_{0,V}=T_VN_{0,V}$ is contained in the Borel subgroup $B=TN_0$ of $\G_2$ with $T_V=T\cap H_{0,V}$ and $N_{0,V}=N_0\cap H_{0,V}$. Moreover, there exists a subtorus $T_0\subset T$ such that $T=T_0\times T_V$. Then we know that $(H,H_V)$ is a good pair for this orbit by Corollary \ref{good pair cor}. This finishes the proof of the proposition and hence the proof of Proposition \ref{norm} and \ref{prop:absConv}.
\end{proof}

\bibliography{G2_period_bib}
\bibliographystyle{amsalpha}
\end{document}